\documentclass[10pt]{amsart}
\usepackage{amsmath,amssymb,amscd,mathrsfs,amscd,inputenc,enumerate,amsfonts,pdfsync,amsbsy}
 
\usepackage[english]{babel}
\usepackage{soul}

\usepackage{color}
\usepackage{graphicx}
\usepackage{extarrows}
\usepackage{tikz-cd}
\newcommand{\qspec}{\mbox{\rm{\texttt{QSpec}}}} 
\newcommand{\qmax}{\mbox{\rm{\texttt{QMax}}}}
\newcommand{\spec}{\mbox{\rm{\texttt{Spec}}}}
\newcommand{\Spec}{\spec}
\newcommand{\Max}{\mbox{\rm{\texttt{Max}}}}

\newcommand{\SStarf}{\mbox{\rm{\texttt{SStar}}}_{f}} 
\newcommand{\insfinss}{\SStarf}

\newcommand{\SStarstabtf}{\widetilde{\mbox{ \rm{\texttt{SStar}}}}}
\newcommand{\inssemistabft}{\SStarstabtf}

\newcommand{\SStar}{\mbox{\rm\texttt{SStar}}}
\newcommand{\inssemistar}{\SStar}

     \DeclareMathOperator{\chiusinv}{\mbox{\rm\texttt{Cl}}^{\mbox{\tiny\rm\texttt{inv}}}}
         \DeclareMathOperator{\chiuscons}{\mbox{\rm\texttt{Cl}}^{\mbox{\tiny\rm\texttt{cons}}}}
     
       \newcommand{\s}{{\mbox{\rm\texttt{s}}}}

\newcommand{\bc}{\color{blue}}%

\newcommand{\ms}{\mathscr}

\newcommand{\Ubold}{\boldsymbol{U}}

\newcommand{\xcal}{{\boldsymbol{\mathcal{X}}}}

\newcommand{\scal}{{\boldsymbol{\mathcal{S}}}}
\newcommand{\ucal}{{\boldsymbol{\mathcal{U}}}}

  \newcommand{\stt} {\widetilde{\star}}
    
    \newcommand{\FF}{\boldsymbol{\overline{F}}}
                \newcommand{\F}{\boldsymbol{F}}
                \newcommand{\f}{\boldsymbol{f}}
    \DeclareMathOperator{\chius}{\mbox{\texttt{Cl}}}
   
     \newcommand{\overr}{\mathtt{Overr}}
     
     \newcommand{\insZ}{\mathbb Z}

\newtheoremstyle{mio}%
	{}{} 
	{\itshape}{} 
	{\bfseries}{.}{ } 
	{#1 #2\thmnote{\mdseries~(\scshape #3)}} 
	
	\theoremstyle{mio}
\newtheorem{teor}{Theorem}[section]
\newtheorem{cor}[teor]{Corollary}
\newtheorem{prop}[teor]{Proposition}
\newtheorem{lemma}[teor]{Lemma}

\theoremstyle{definition}
\newtheorem{defin}[teor]{Definition}
\newtheorem{ex}[teor]{ \textbf{Example}}
\newtheorem{oss}[teor]{Remark}

\DeclareMathOperator{\Cl}{Cl}
\DeclareMathOperator{\rad}{rad}

\begin{document}  

\title[Semigroup primes of a commutative ring]
{Topological properties of\\ semigroup primes  of a commutative ring}
\author{Carmelo A. Finocchiaro,
Marco Fontana, and 
Dario Spirito}

\noindent \email{\newline $\;\;\;\;$ finocchiaro@math.tugraz.at; fontana@mat.uniroma3.it; spirito@mat.uniroma3.it}

\address{C.A.F.: Institute of Analysis and Number Theory, University of Technology,
\newline $\;\;\;\;$ Steyrergasse 30/II, 8010 Graz, Austria}

\address{M.F.{\&}D.S.: Dipartimento di Matematica e Fisica, Universit\`a degli Studi
``Roma Tre'', \newline $\;\;\;\;$ Largo San Leonardo Murialdo, 1, 00146 Roma, Italy}



\thanks{This work was partially supported by {\sl GNSAGA} of {\sl Istituto Nazionale di Alta Matematica}. The first named author was also supported by a Post Doc Grant from the University of Technology of Graz (Austrian Science Fund (FWF): P 27816).}
\date{\today}

\begin{abstract} 
A  semigroup prime  of a commutative ring $R$ is a prime ideal of the semigroup $(R,\cdot)$.  One of the purposes of this paper is to study, from a topological point of view, the space $\scal(R)$ of prime semigroups of $R$.  We show that, under a natural topology introduced by B. Olberding  in 2010, $\scal(R)$ is a spectral space  (after Hochster),   spectral extension  of $\Spec(R)$, and that the assignment $R\mapsto\scal(R)$  induces a contravariant functor.   We  then relate -- in the case   $R$ is an integral domain -- the topology on $\scal(R)$ with the Zariski topology on the set of overrings of $R$.   Furthermore,  we investigate the relationship between $\scal(R)$ and the space $\boldsymbol{\mathcal{X}}(R)$ consisting of all nonempty inverse-closed subspaces of $\spec(R)$, which has been introduced and studied in \cite{FiFoSp-X(X)}.
 In this context, we 
show that $\scal( R)$   is a spectral retract of $\boldsymbol{\mathcal{X}}(R)$    and we characterize when   $\scal( R)$   is canonically homeomorphic to $\boldsymbol{\mathcal{X}}(R)$, both in general and when $\spec(R)$ is a Noetherian space. In particular, we obtain that, when $R$ is a B\'ezout domain,   $\scal( R)$   is  canonically homeomorphic both to $\boldsymbol{\mathcal{X}}(R)$ and to the space $\overr(R)$ of the overrings of $R$ (endowed with the Zariski topology).
 Finally, we compare the space $\boldsymbol{\mathcal{X}}(R)$ with the space $\scal(R(T))$ of semigroup primes of the Nagata ring $R(T)$, providing a canonical spectral embedding $\xcal(R)\hookrightarrow\scal(R(T))$ which makes $\xcal(R)$ a spectral retract of $\scal(R(T))$.

 \end{abstract}

\maketitle

\section{Introduction and preliminaries}
 
The concept of prime ideal, and the closely related concept of localization,    play  a   fundamental role in commutative ring theory. In the forties of the last century,  the concept of prime ideal was introduced  in the setting of  semigroups, and some analogies and differences between the ring and semigroup theories were  pointed out  (cf., for instance, \cite{re-40}, \cite{gri}, and \cite{ki}).  Since a ring $R$ can be also   regarded  as  a semigroup (by considering only the multiplicative structure), it is reasonable to bring back the concept   of semigroup prime  from semigroups to rings: hence, we define a \emph{semigroup prime} of a ring $R$ to be a prime ideal of the semigroup $(R,\cdot)$.

Clearly, every prime ideal of $R$ is also a semigroup prime, but not conversely:  the set $\scal(R)$ of all semigroup primes of $R$ is in general much larger than the prime spectrum $\spec(R)$ of $R$. 
An additional link ties the two concepts:
semigroup primes of $R$ turn out to be the complement of saturated multiplicatively closed subsets of $R$ and so they give rise to general ring of fractions,  while prime ideals give rise to localizations. 

  Nevertheless,  for   a long time, semigroup primes of a commutative ring were left out from the   mainstream of    investigation, even in the natural context of multiplicative ideal theory of rings and integral domains.

Recently, B. Olberding \cite{ol} has  considered the space   $\scal( R)$, equipped with a Zariski-like topology, for obtaining new important properties of the spaces of overrings and 
valuation overrings of an integral domain $R$.

In this paper, we pursue the study of $\scal( R)$, mainly from a topological point of view, considering the general case of a commutative ring $R$ with applications to the special case of when $R$ is an integral domain.  The relevant topologies that turn out to be useful in our investigation are the   hull-kernel topology (classically introduced by Stone \cite{st-37}) or Zariski topology, the constructible  or patch topology (cf. \cite{EGA}, and \cite{ho}), with an underlying ultrafilter theoretic approach (cf. \cite{fo-lo-2008}, \cite{Fi} and 
\cite{lo-2011}) and the inverse topology introduced by Hochster on arbitrary spectral spaces   \cite{ho} (definitions  and properties used in the present paper will be recalled later in this section).

As a starting point, we prove that $\scal( R)$, when endowed with the hull-kernel topology, is a new unconventional example of spectral space (after Hochster), that the  inclusion map $\spec(R) \hookrightarrow \scal( R)$ is a spectral map, and that the assignment $R \mapsto \scal(R)$ induces a contravariant functor.  
Next, we compare the spectral space $\scal( R)$  with the space $\boldsymbol{\mathcal{X}}(R)$ consisting of all   nonempty inverse-closed subspaces of $\spec(R)$, which has been  
introduced and studied in  \cite{FiFoSp-X(X)} to classify, from a topological point of view,  distinguished classes of Krull closure operations, namely the  e.a.b. semistar operations and the stable semistar operations of finite type. In particular, we prove here that $\scal( R)$   is a spectral retract of $\boldsymbol{\mathcal{X}}(R)$ (Proposition \ref{homeo}) and we characterize when   $\scal( R)$   is canonically homeomorphic to $\boldsymbol{\mathcal{X}}(R)$, both in general and when $\spec(R)$ is a Noetherian space. In the general case, this happens  under the purely algebraic condition that  the radical of every finitely generated ideal of $R$ is the radical of a principal ideal (Theorem \ref{prop:scal-xcal}) and, in the Noetherian space case,  when every prime ideal of $R$ is the radical of a principal ideal (Corollary \ref{j-homeo}). When $R$ is a B\'ezout domain, we prove that  $\scal( R)$   is  canonically homeomorphic both to $\boldsymbol{\mathcal{X}}(R)$ and to the space $\overr(R)$ of the overrings of $R$ endowed with the Zariski topology (Corollary \ref{prop:QR}).  
When $R$ is a Dedekind domain, $\scal( R)$   is  canonically homeomorphic to $\boldsymbol{\mathcal{X}}(R)$ if and only if the ideal class group of $R$ is torsion (Remark \ref{rk-dedekind}).
Each of the previous homeomorphisms can be   interpreted  as  a topological ``dual'' statement   to  Hilbert's Nullstellensatz, providing 
a one-to-one correspondence, compatible with the natural orders,
 between inverse-closed subspaces of $\spec(R)$ and semigroup  primes  of $R$.

In the final section, we compare the space $\boldsymbol{\mathcal{X}}(R)$ with the space $\scal(R(T))$ of semigroup primes of the Nagata ring $R(T)$  (where $T$ is an 
indeterminate over $R$).
 In particular, we provide a canonical spectral embedding  $\boldsymbol{\mathcal{X}}(R) \hookrightarrow \scal(R(T))$  which   makes $\boldsymbol{\mathcal{X}}(R)$ a spectral retract of $\scal(R(T))$ (Propositions 4.2 and 4.4).

\medskip

 In order to facilitate the reader, we recall next some preliminary notions and results that will be used in the present paper.

\subsection{Spectral spaces}

A topological space is \emph{spectral}  (after M. Hochster \cite{ho}) if it is homeomorphic to the prime spectrum of a (commutative) ring. While defined in algebraic terms, this concept admits a purely topological characterization:   a topological space $X$ 
is spectral if and only if it is T$_0$, quasi-compact, it admits a basis of open and quasi-compact subspaces that is closed under finite intersections, and every irreducible closed  subset of $X$ has a (unique) generic point (i.e., it is the closure of a one-point set) \cite{ho}.
  If $X$ and $Y$ are spectral spaces, a {\it  spectral map} $f: X\rightarrow Y$ is a map such that $f^{-1}(U)$ is a quasi-compact open subspace of $X$, for each quasi-compact open subspace $U$ of $Y$; spectral maps are the morphisms in the category having the spectral spaces as objects. 

It is well known that the prime spectrum of a commutative ring endowed with the Zariski topology  is always $T_0$,
but almost never Hausdorff (it is Hausdorff   if and only if the ring has Krull dimension zero).
Thus, many authors have considered a finer topology on the prime spectrum of a ring, known as the \emph{constructible topology}   \cite [pages 337-339]{EGA}  or as the \emph{patch topology} \cite{ho}.

As in \cite{sch-tr}, we introduce  the \emph{constructible topology}  by a Kuratowski closure ope\-rator: if    $ X$ is a spectral space, for each subset $Y$  of $X$, we set:
$$
\begin{array}{rl}
\chiuscons(Y) \! :=  \!\bigcap \{U \!\cup\! (X\!\setminus\! V) \mid  & \hskip -5pt \mbox{ $U$ and $V$ open and quasi-compact in }  X,  \\
 & \hskip -4pt U  \!\cup \! (X \!\setminus \! V) \supseteq Y \}\,.
 \end{array}
$$
We denote by $X^{\mbox{\tiny\texttt{cons}}}$ the set $X$, equipped with the constructible topology.
For Noetherian topological spaces, the closed sets of this topology coincide with the ``constructible sets'' classically defined in \cite{ch}. 
It is well known that $X^{\mbox{\tiny\texttt{cons}}}$ is a spectral space and that the constructible topology  is a refinement of the given topology which is always Hausdorff.   

\subsection{The inverse topology on a spectral space}
Recall that the given topology on  a spectral space $X$ induces a canonical partial order  $ \leq_X$, denoted simply by $\leq$ when no danger of confusion can arise, defined by $x\leq_X y$ if $y\in \chius(\{x\})$, for $x,y\in X$, where $\chius(Y)$ denotes the closure of a subset $Y$ of $X$. The set
$
Y^{\mbox{\tiny\texttt{gen}}} :=\{x\in X \mid  y\in \chius(\{x\}),\mbox{ for some }y\in Y \}
$
is called \textit{closure under generizations of $Y$}. 
Similarly, using the opposite order,  the set
$
Y^{\mbox{\tiny\texttt{sp}}}:=\{x\in X \mid  x\in \chius(\{y\}),\mbox{ for some }y\in Y \}
$
is called \textit{closure under specia\-li\-zations of $Y$}. We say that $Y$ is \textit{closed under generi\-zations}  (respectively, \textit{closed under specia\-li\-zations})  if $Y=Y^{\mbox{\tiny\texttt{gen}}}$ (respectively, $Y=Y^{\mbox{\tiny\texttt{sp}}}$). 
It is straightforward that, for two elements $x, y$ in a spectral space $X$, we have: 
 $$
 x \leq y \quad \Leftrightarrow \quad \{x \}^{\mbox{\tiny\texttt{gen}}} \subseteq \{y \}^{\mbox{\tiny\texttt{gen}}} \quad \Leftrightarrow \quad\{x \}^{\mbox{\tiny\texttt{sp}}} \supseteq \{y \}^{\mbox{\tiny\texttt{sp}}}\,.
 $$

 Given  a spectral space $X$, Hochster \cite[Proposition 8]{ho} introduced a new topology on $ X$, that we call here the  {\it inverse topology}, by defining a Kuratowski closure operator, for each subset $Y$  of $X$, as follows:
  $$
 \chius^{\mbox{\tiny\texttt{inv}}}(Y) :=
\bigcap \{U   \mid  \mbox{ $U$ open and quasi-compact in } X,  \; 
U  \supseteq Y \}\,.
$$
If we denote by  $X^{\mbox{\tiny\texttt{inv}}}$ the set $X$ equipped with the inverse topology, Hochster proved that $X^{\mbox{\tiny\texttt{inv}}}$ is still a spectral space and the partial order on $X$ induced by the inverse topology is the opposite order of that induced by the given topology on $X$     \cite[Proposition 8]{ho}. 
In particular, the closure under generizations $\{x\}^{\mbox{\tiny\texttt{gen}}}$ of a singleton is closed in the inverse topology  of $X$, since   $\{x \}^{\mbox{\tiny\texttt{gen}}}=\bigcap\{U \mid  U \subseteq X\mbox{ quasi-compact and open},\, x\in U \}$.  
On the other hand, it is trivial, by the definition,  that the closure under specializations of a singleton $\{x \}^{\mbox{\tiny\texttt{sp}}}$ is closed in the given topology of $X$, since $\{x \}^{\mbox{\tiny\texttt{sp}}}= \chius(\{x\})$.

Finally, recall that,  by \cite[Remark 2.2]{fifolo2}, we have 
$\chiusinv(Y)=(\chiuscons(Y))^{\mbox{\tiny\texttt{gen}}}$. 
It follows that each closed set in the inverse topology  (called for short, \emph{inverse-closed}) is closed under generizations  and, from \cite[Proposition 2.6]{fifolo2},  that a quasi-compact subspace $Y$ of $X$ closed for generizations is inverse-closed.  
On the other hand, the closure of a subset $Y$ in the given topology of $X$, $\chius(Y)$, coincides with   $(\chiuscons(Y))^{\mbox{\tiny\texttt{sp}}}$   \cite[Remark 2.2]{fifolo2}.

\subsection{The spectral space of the inverse-closed subspaces}\label{X(X)}

  Given a spectral space $X$, let $\boldsymbol{\mathcal{X}}(X) := \{ Y \subseteq X \mid   Y \neq \emptyset,\,  Y =\chiusinv(Y)\}$, that is, $\xcal(X)$ is the set of all nonempty subset of $X$ that are closed in the inverse topology.  
   If $X=\spec(R)$ for some ring $R$, we write for short $\xcal(R)$ instead of $\xcal(\spec(R))$.

 We define a \textit{Zariski topology on} $\boldsymbol{\mathcal{X}}(X)$ by taking, as subbasis (in fact, a basis) of open sets, the sets of the form 
$$
\boldsymbol{\mathcal{U}}(\Omega):=\{Y\in \boldsymbol{\mathcal{X}} \mid Y\subseteq \Omega \},$$
where $\Omega$ varies among the quasi-compact open subspaces of $X$. 
Note that $\emptyset \neq \Omega \in \boldsymbol{\mathcal{U}}(\Omega)$, since a quasi-compact open subset $\Omega$ of $X$ is a closed in the inverse topology of $X$. 
Note also that,  when $X=\spec( R)$, for some ring $R$, a generic basic open set of the Zariski topology on $\boldsymbol{\mathcal{X}}(R)$ is of the form
$$
\boldsymbol{\mathcal{U}}(\texttt{D}(J))=\{Y\in \boldsymbol{\mathcal{X}}(R) \mid Y\subseteq \texttt{D}(J) \},
$$
where $J$ is any finitely generated ideal of $R$, and, as usual,
 $$
 \texttt{V}(J):= \{ P\in \spec(R) \mid J  \subseteq P\} \quad \mbox{   and   } \quad \texttt{D}(J):=  \spec(R) \setminus \texttt{V}(J)\,.
 $$

It was proved in \cite[Theorem 3.2]{FiFoSp-X(X)} that: 

\begin{itemize}
\item[(1)]
the space
 $ \boldsymbol{\mathcal{X}}(X)$, 
 endowed with the Zariski topology, is a spectral space;

 \item[(2)] the canonical map $\varphi:X\hookrightarrow  \boldsymbol{\mathcal{X}}(X)$, defined by
$
\varphi(x):=\{x\}^{\mbox{\tiny{\rm{\texttt{gen}}}}}
$, for each $x\in X$,   is  a spectral embedding  (and, in particular, an order-preserving embedding  between ordered sets, with the ordering induced by the Zariski topologies). 
\end{itemize}

\subsection{Semistar operations}

Let $D$ be an integral domain with quotient field $K$. Let $\FF(D)$ (respectively, $\F(D)$; $\f(D)$) be the set of all nonzero $D$--submodules of $K$ (respectively, nonzero fractional ideals; nonzero finitely 
generated fractional ideals) of $D$ (thus, $\f(D)\subseteq\F(D)\subseteq\FF(D)$).

A mapping $\star:\FF(D)\longrightarrow\FF(D)$, $E\mapsto E^\star$, is called a \emph{semistar operation} of $D$ if, for all $z\in K$, $z\neq 0$ and for all $E,F \in\FF(D)$, the following properties hold: $\mathbf{\bf(\star_1)} \;(zE)^\star =zE^\star$; $\mathbf{\bf (\star_2)} \; E\subseteq F \Rightarrow E^\star \subseteq F^\star$; $\mathbf{ \bf (\star_3)}\; E \subseteq E^\star$; and $\mathbf{ \bf (\star_4)}\;  E^{\star \star} := (E^\star)^\star = E^\star $.   We denote the set of all semistar operations on $D$ by $\SStar(D)$.

Given a semistar operation $\star$ on $D$, a nonzero ideal $I$ of $D$ is called a \emph{quasi-$\star$-ideal} if $I = I^\star \cap D$. A \emph{quasi-$\star$-prime} is a quasi-$\star$-ideal which is also a prime ideal. The set of all quasi-$\star$-prime ideals of $D$ is denoted by $ \qspec^\star(D)$. The set of maximal elements in the set of proper quasi-$\star$-ideals of $D$ (ordered by set-theoretic inclusion) is denoted by $\qmax^\star(D)$ and it is a subset of $\qspec^\star(D)$.

A semistar operation $\star$ is \emph{of finite type} if, for every $E\in\FF(D)$,
\begin{equation*}
E^{\star}=\bigcup \{F^\star \mid  F \subseteq E,  F \in \f(D)\}.
\end{equation*}
 It is well known that if $\star$ is a semistar operation of finite type then $\qmax^\star(D)$ is nonempty \cite[Lemma 2.3(1)]{fo-lo-2003}. 

For more details on semistar operations see, for instance, \cite{ep-12}, \cite{ep-15}, \cite{HK-01}, \cite{HK-2011}, \cite{ma-2011}, and \cite{ok-ma-2006}; for  the case of star operations see, for instance,  \cite{dda-88}, \cite{dda-dfa-90}, \cite{dda-cl-05},  \cite{el-2010} and    \cite{gi}.
 
 The  set of all semistar operations of finite type is denoted by $\SStarf(D)$.  

In \cite{FiSp}, the set $\SStar(D)$ of all semistar operation was endowed with a topology (called the \emph{Zariski topology}) having, as a subbasis of open sets, the sets of the type 
\begin{equation*}
\texttt{V}_E:=\{\star\in\SStar(D)\mid 1\in E^\star\}, \mbox{ where $E$ is a nonzero $D$-submodule of $K$}.
\end{equation*}
This topology makes $\SStar(D)$ into a quasi-compact T$_0$ space, and $\SStarf(D)$ into a spectral space.

\subsection{Spectral semistar operations}
Let $D$ be a domain and $Y\subseteq\Spec(D)$ be nonempty. The semistar operation $\s_Y$ is defined as the map such that 
\begin{equation*}
E^{\mbox{\small${\s}$}_Y}=\bigcap\{ED_P\mid P\in Y\}\quad\text{for every $E\in\FF(D)$}.
\end{equation*}
The semistar operations on $D$ that can be written as $\s_Y$, for some $Y$, are called \emph{spectral}; the set of all finite  type spectral semistar operations, denoted by $\inssemistabft(D)$, is a spectral space \cite[Theorem 4.6]{FiFoSp-JPAA}. By \cite[Corollary 4.4]{FiSp}, $\s_Y$ is of finite type if and only if $Y$ is quasi-compact,   as a subspace of $\spec(D)$, endowed with the Zariski topology (see also  \cite{fohu} and  \cite{HK-01}).

There is a canonical map
\begin{equation*}
\begin{aligned}
\widetilde{\Phi}\colon\inssemistar(D) & \longrightarrow \inssemistabft(D)\\
\star & \longmapsto \stt,
\end{aligned}
\end{equation*}
where $\stt$ is defined as the map such that, for every $E\in\FF(D)$,
\begin{equation*}
\begin{array}{rcl}
E^{\stt}:= & \bigcup \{ (E:J)\mid J \mbox{ nonzero finitely generated ideal of } D \\
& \mbox{ such that } J^\star =D^\star\}.
\end{array}
\end{equation*}
The map $\widetilde{\Phi}$ is a topological retraction \cite[Proposition 4.3(2)]{FiFoSp-JPAA}; in particular, $\star=\stt$ if and only if $\star$ is spectral and of finite type \cite[Corollary 3.9(2)]{fohu}.

The space $\SStarstabtf(D)$ can also be seen as a natural ``extension'' of $\spec(D)$, since the canonical map $\s\colon\spec(D)\hookrightarrow\inssemistabft(D) $, defined by 
$P  \mapsto \s_{\{P\}}$, is a topological embedding.

An alternative way to see the space $\inssemistabft(D)$ is through the space $\xcal(D)$ recalled in Section \ref{X(X)}.   By  \cite[Proposition 5.2]{FiFoSp-X(X)}, we have the following.
\begin{itemize}
\item The map
$\boldsymbol{\s^\sharp}\colon\xcal(D)  \rightarrow \inssemistabft(D)$, defined by
$Y   \mapsto \s_Y$, 
and the map
$\Delta \colon$ $ \inssemistabft(D)   \rightarrow\xcal(D) $, defined by
$\star \mapsto \qspec^\star(D)$,
are homeomorphisms and  are inverse of each other. 
\item
If $\varphi:\spec(D) \hookrightarrow \xcal(D)$ is the canonical embedding defined in \ref{X(X)}(2), then $\boldsymbol{\s^\sharp} \circ \varphi = \s$.
\end{itemize}

\begin{oss}\label{qspec-inv}
Let $\star $ be a semistar operation of finite type on the integral domain $D$.
 It is well known that $\qmax^\star(D) = \qmax^{\stt}(D)$ and  $\widetilde{\star} =\s_{\tiny \qspec^{\star}\!(D)} = \s_{\tiny \qmax^{\star}\!(D)}  = \s_{\tiny \qmax^{\stt}(D)} $
 \cite[Lemma 2.4 and Corollaries 2.7 and 3.5]{fo-lo-2003}.
  Moreover, since $\qspec^{\widetilde\star}(D)$ is closed in the inverse topology of $\spec(D)$ and the maps $\Delta,\s^\sharp$ are homeomorphisms (see above), it follows  that $\chiusinv(\qspec^\star(D)) =\qspec^{\widetilde\star}(D)$.
  Therefore,  by  \cite[Proposition 5.8]{FiSp}, we also 
  have $$\widetilde{\star}=\s_{\tiny\chiusinv(\qspec^\star(D))} =\s_{\tiny\qspec^{\widetilde\star}(D)}.$$
\end{oss}

\vskip -15pt

\subsection{The set of overrings of an integral domain }\label{sect:overrings}

Let $\overr(D)$ be the set of all overrings of $D$, endowed with the topology whose   basic open sets are of the form {\rm\texttt{B}}$(x_1, x_2, \dots, x_r)  := \overr(D[x_1, x_2,\ldots,x_n])$,  for $x_1, x_2, \ldots,x_n$  varying in $ K$ \cite[Ch. VI, \S 17]{zs}. For recent investigations on topological spaces of  overrings of an integral domain see, for instance,  \cite{fi-fo-sp-survey}, \cite{FiFoSp-JPAA}, \cite{ol}, \cite{ol-11}, \cite{ol-15-pacjmath}, \cite{ol-15-jppa}.

It is known that:
\begin{itemize}
\item[\rm(1)] The topological space $\overr(D)$ is a spectral space \cite[Proposition 3.5]{Fi} and the map $\iota:\overr(D)\hookrightarrow\SStarf(D)$, defined by $\iota(T) := \wedge_{\{T\}}$, for each $T \in \overr(D)$, is a topological embedding  \cite[Proposition 2.5]{FiSp}.
\item[\rm(2)] The map $\pi: \SStarf(D)\rightarrow\overr(D)$, defined by $\pi(\star):=D^\star$,  for any $\star\in \SStarf(D)$, is a  topological retraction \cite[Proposition 3.2]{fi-fo-sp-survey}. 
\end{itemize}

\section{The space of semigroup primes}

Let $R$ be a ring. The purpose of the present section is to investigate a   natural spectral extension of $\spec({R})$ which is intermediate between $\spec({R})$  and $\xcal({R})$, namely the embeding of the prime spectrum into the set of semigroup primes.

Using the terminology of \cite{ol},  we   recall the following definition:

\begin{defin}
A \emph{semigroup prime} is a nonempty proper subset $\mathscr{Q}$ of a ring $R$ such that:
\begin{enumerate}
\item[(a)]  for each $r \in R$ and for each $\pi \in \mathscr{Q}$, $r\pi \in \mathscr{Q}$;
\item[(b)] for all $\sigma,\ \tau \in R \setminus \mathscr{Q}$,  $\sigma\tau \in R \setminus \mathscr{Q}$.
\end{enumerate}
\end{defin}
 
Obviously, every prime ideal of $R$ is also a semigroup prime of $R$. More generally,  if  $Y$ is a nonempty collection of prime ideals of $R$, then  $\mathscr{P}(Y):=\bigcup \{P \in   \spec(R) \mid P\in Y\}$ is a semigroup prime of $R$.
A more precise result is given next.

\begin{lemma}\label{semigroupprime} Let $\mathscr{Q}$  be a proper subset  of a ring $R$. Then, $\mathscr{Q}$ is a semigroup prime  of $R$  if and only if there exists a nonempty collection of prime ideals $Y$ of $R$ such that
$ \mathscr{Q} = \mathscr{P}(Y)$.
\end{lemma}
\begin{proof} We just need to prove the ``only if'' part.
 For each   semigroup prime $\mathscr{Q}$ of $R$,  $R \setminus \mathscr{Q}$ is a multiplicatively closed subset of $R$ and it is also saturated, since if $\alpha\beta \in R \setminus \mathscr{Q}$ 
then, from (a) of the previous definition, it follows immediately that both  $\alpha$ and $\beta$ belong to $ R \setminus \mathscr{Q}$.  Since a saturated multiplicatively closed set is the complement of the union of prime ideals  \cite[Theorem 2]{kap},  if $Y$ is a nonempty set of prime ideals of $R$ such that $ R \setminus \mathscr{P}(Y)$ coincides with the saturated multiplicatively closed set $R \setminus \mathscr{Q}$, then $ \mathscr{Q} = \mathscr{P}(Y)$.
 \end{proof}
 
Let $\scal( R) := \{\mathscr{Q} \mid \mathscr{Q} \mbox{ is a semigroup prime of } R\}$. As in \cite[(2.3)]{ol}, the set 
$\scal( R)$ can be endowed with a {\it hull-kernel topology}, by taking   as a basis for the open sets the subsets
$$
\boldsymbol{U}(x_1, x_2, \dots, x_n) := \{ \mathscr{Q} \mid x_i \notin \mathscr{Q} \mbox{ for some }  i,\ 1\leq i \leq n \}\,,
$$
where $x_1, x_2, \dots, x_n \in R$.

\begin{prop}\label{embedd}
Let $R$ be a ring. 
\begin{enumerate}
\item[{\rm (1)}] The set $\scal( R)$ of semigroup primes of $R$  with the  hull-kernel topology is a spectral space.
\item[{\rm (2)}]  The collection of sets $\{\boldsymbol{U}(x)\mid x\in R \}$ is a basis of open and quasi-compact subspaces of $\scal(R)$.
\item[{\rm (3)}]  The  set theoretic inclusion $i: \spec( R)\hookrightarrow \scal(R)$  is a spectral embedding.

\end{enumerate}
\end{prop}

\begin{proof}
(1) Since $R \setminus \mathscr Q$ is a saturated multiplicative set of $R$  for each  $\mathscr Q\in \scal( R)$, then  $\boldsymbol{U}(xy)=\boldsymbol{U}(x) \cap \boldsymbol{U}(y)$ for each pair $x, y \in R$. 
By definition, it follows easily that a basis of open sets for $\scal(R)$ is given by $\{\boldsymbol{U}(x) \mid x\in R \}$.

By \cite[Corollary 3.3]{Fi}, to show that $\scal(R)$ is a spectral space it suffices to show that, for any ultrafilter $\ms U$ on $\scal(R)$, the set
$$
\{\mathscr Q\in \scal(R) \mid \forall x\in R,\, \;  \mathscr Q\in \boldsymbol{U}(x)\Leftrightarrow\boldsymbol{U}(x)\in\ms U \,  \}
$$
is nonempty. 
Set $\mathscr Q_{\ms U}:=\{r\in R \mid \scal(R)\setminus \boldsymbol{U}(r)\in \ms U \}$. 
An easy argument shows that $\mathscr Q_{\ms U}$ is a semigroup prime of $R$. 
 Moreover,  by definition,  for each $x\in R$, $\mathscr Q_{\ms U}\in \boldsymbol{U}(x)$ if and only if $\boldsymbol{U}(x)\in \ms U$. 

(2) By \cite[Propositions 2.11, 3.1(3,b) and 3.2]{Fi}, the sets $\boldsymbol{U}(x)$ are clopen, with respect to the constructible topology of $\scal(R)$ and, a fortiori, they are quasi-compact with respect to the hull-kernel topology. 
 
(3) 
 The conclusion follows from the fact that  the  hull-kernel topology of $\scal( R)$ induces the Zariski topology on $\spec( R) $, since $i^{-1}(\boldsymbol{U}(x)) =
\boldsymbol{U}(x) \cap \spec( R) = \texttt{D}(x)$ and  from the fact   that $i(\texttt{D}(x)) = \boldsymbol{U}(x) \cap i(\spec( R))$, for each $x \in R$.
\end{proof}

\begin{oss}
Let $S$ be a semigroup. A \emph{prime ideal} of $S$ is a nonempty  proper subset $I\subseteq S$ such that $xs\in I$ for every $x\in I$, $s\in S$ and such that $st  \in S\setminus I$ for every $s,t  \in S\setminus I$   (see, for example, \cite{gri,ki}).  Under this terminology, a prime semigroup of a ring $R$ is just a prime ideal of the multiplicative semigroup $(R,\cdot)$.

The topology we introduced above in the case of prime semigroups of a ring can be extended naturally to the set $\scal(S)$ of the prime ideals of the semigroup $S$; likewise, the proof of Proposition \ref{embedd}(1) can be transferred verbatim to the case of semigroups, showing  the slightly more general result  that $\scal(S)$ is a spectral space. 
\end{oss}
  
  \begin{oss}   The subspace  $\spec( R)$ of $\scal({R})$ is  dense  in $\scal({R})$. 
  In fact, the closure of $\spec( R)$ is the set of all   $\mathscr Q \in\scal({R})$  containing the nilradical of $R$, which is  $\scal({R})$  (since each $\mathscr Q$ contains at least one prime $P \in \spec( R)$). 

Following \cite[D\'efinition (2.6.3)]{EGA}, recall that a subset $X_0$ of a topological space $X$ is said to be \emph{very dense in $X$} if, for any open sets $U,V\subseteq X$, the equality $U\cap X_0=V\cap X_0$ implies $U=V$, that is, in our setting, if the map $\boldsymbol{U} \mapsto  \boldsymbol{U}\cap \spec( R)$, from the open subsets of $\scal({R})$ to the open subsets of $\spec( R)$, is injective. Under this terminology, $ \spec( R)$  is  not very dense  in\ $\scal({R})$.
  For instance, consider a 1-dimensional B\'ezout domain $D$ with exactly two maximal ideals, say $M$ and $N$. Then, $\scal(D)$ has a maximal element (namely $M\cup N$) that is a closed point but does not belong to $\spec(D)$.
\end{oss}

  Given a ring homomorphism $f:R_1 \rightarrow R_2$, we can canonically associate  to $f$ a map
\begin{equation}\label{eq:scalf}
\begin{aligned}
\scal(f)\colon\scal(R_2) & \longrightarrow \scal(R_1)\\
\mathscr{Q} & \longmapsto f^{-1}(\mathscr{Q}).
\end{aligned}
\end{equation}
We investigate   next  the properties of this map.

\begin{prop}\label{prop:scal-functor}
Let $f:R_1\rightarrow R_2$ be a ring homomorphism, let $\scal(f)$ be the map defined above and let $f^a:\spec(R_2)\rightarrow\spec(R_1)$ be the   continuous map canonically associated to $f$.  Assume that $\scal(R_1)$ and $\scal(R_2)$ are endowed with  the hull-kernel topology. Then:
\begin{enumerate}
\item[{\rm (1)}]\label{prop:scal-functor:wd} $\scal(f)$ is well-defined,   (continuous) and spectral;
\item[{\rm (2)}]\label{prop:scal-functor:diagram} if $i_k:\Spec(R_k)\longrightarrow\scal(R_k)$ is the set-theoretic inclusion    ($k=1,2$), then $\scal(f)\circ i_2=i_1\circ f^a$;
\item[{\rm (3)}]\label{prop:scal-functor:functor} the assignment $R\mapsto\scal(R)$, $f\mapsto\scal(f)$, is a functor from the category of rings to the category of spectral spaces.
\end{enumerate}
\end{prop}
\begin{proof}
(1)   Let $\mathscr{Q}$ be a semigroup prime of  $ R_2$,  let $r\in R_1$ and $\pi\in f^{-1}(\mathscr{Q})$. Then, $f(\pi r)=f(\pi)f(r)\in f(r)\mathscr{Q}\subseteq\mathscr{Q}$, so that $r\pi\in f^{-1}(\mathscr{Q})$; moreover, if $\sigma,\tau\notin f^{-1}(\mathscr{Q})$, then $f(\sigma),f(\tau)\notin\mathscr{Q}$ and thus $f(\sigma)f(\tau)\notin\mathscr{Q}$, that is, $\sigma\tau\notin f^{-1}(\mathscr{Q})$. Hence, $\scal(f)$ is well-defined. Moreover, $\scal(f)^{-1}(\Ubold(x))=\Ubold(f(x))$   for each $x \in R_1$,  and thus $\scal(f)$ is continuous.
 By the last part of  Proposition \ref{embedd}(1),   the collection $\{\Ubold(y)\mid y\in A\}$ is a basis of  quasi-compact subsets of $\scal(A)$, for any ring $A$.   Thus,  the previous  reasoning implies that $\scal(f)$ is a spectral map.

(2) is straightforward.

(3) follows from the previous points and the fact that,  given two ring homomorphisms  $f:R_1\rightarrow R_2$ and $g:R_2\rightarrow R_3$, $\scal(g\circ f)=\scal(f)\circ\scal(g)$,  which is a direct consequence of the definitions.
\end{proof}

 We now start the study of the relationship between the spectral spaces $\scal({R})$  and $\xcal({R})$.

\begin{prop}\label{embedd-II} 
Let $R$ be a ring. 
\begin{enumerate}
 \item[{\rm (1)}] For each $\mathscr{Q} \in \scal( R)$, set $\Sigma_{\mathscr{Q} } := R \setminus\mathscr{Q} $ and $R_{\mathscr{Q}} := \Sigma_{\mathscr{Q} }^{-1}R$.
The map 

\begin{equation*}
\begin{aligned}
j\colon\scal(R) & \longrightarrow\xcal(R)\\
\mathscr{Q} & \longmapsto \lambda^a(\spec(R_{\mathscr{Q}})),
\end{aligned}
\end{equation*}
where $\lambda^a:\spec(R_{\mathscr{Q}})\rightarrow\spec(R)$ is the spectral map associated to the localization homomorphism $\lambda:R\rightarrow R_{\mathscr{Q}}$, is a topological embedding. Moreover, $j(\mathscr{Q}) = \{ P \in \spec({R}) \mid P \subseteq \mathscr{Q}\}$, for each $  \mathscr{Q} \in \scal({R})$.

\item[{\rm (2)}]  The   canonical spectral embedding $\varphi: \spec( R) \hookrightarrow \xcal( R)$  \cite[Theorem 3.3(2)]{FiFoSp-X(X)} coincides with $j\circ i$.
\end{enumerate}
\end{prop}
\begin{proof}
(1) The map
$j$ is clearly injective. In order to prove that $j$ is continuous we have to verify that, given a  nonzero  finitely generated ideal $J$ of $R$, then
$$
H:=j^{-1}(\ucal(\texttt{D}(J)))=\{\ms Q\in \scal (R)\mid j(\ms Q)\subseteq \texttt{D}(J) \}
$$ 
is open in $\scal(R)$. 
Take a point $\ms Q\in H$ and assume that $J \subseteq \ms Q$. Then $J$ is disjoint from $\Sigma_{\mathscr{Q} } $, and thus there exists a prime ideal $P$ of $R$ disjoint from $\Sigma_{\mathscr{Q} } $ and such that $J \subseteq P$.
 On the other hand, keeping in mind that $\ms Q\in H$ and $P \cap \Sigma_{\mathscr{Q} }=\emptyset$, we have $P \in j(\ms Q)\subseteq \texttt{D}(J)$, contradiction. 
 This shows that $J \nsubseteq \ms Q$, and thus there exists an element $x\in J\setminus\ms Q$. 
 It follows that $\ms Q\in \boldsymbol{U}(x)$ and,  moreover,  $\boldsymbol{U}(x) \subseteq H$.
  Since $\{\boldsymbol{U}(x)\mid x\in R \}$ is a basis of open sets for $\scal(R)$, it follows that $H$ is open and $j$ is continuous. 
  Now, the fact that $j$ is a topological embedding follows immediately from the equality $j(\boldsymbol{U}(x))=j(\scal(R))\cap \ucal(\texttt{D}(x))$ that holds for each $x\in R$. 
  
For the last statement, we have $P \in \lambda^a(\spec(R_{\mathscr{Q}}))$ if and only if $ P \cap  \Sigma_{\mathscr{Q}} = \emptyset$, i.e., if and only if $P \subseteq \mathscr{Q}$.

(2) is a straightforward consequence of the definitions.
\end{proof}

\begin{prop}\label{prop:scalf}
Let $f:R_1\rightarrow R_2$ be a ring homomorphism,  $f^a:\spec(R_2)\rightarrow\spec(R_1)$  the associated map of spectra,  $\scal(f)$ the map defined in \eqref{eq:scalf},
$\xcal(f^a): \xcal(R_2) \rightarrow \xcal(R_1)$ the spectral map defined in   \cite[Proposition 4.1]{FiFoSp-X(X)} and let $i_k: \spec(R_k) \rightarrow \scal(R_k)$ (respectively, $j_k: \scal(R_k) \rightarrow \xcal(R_k)$) the spectral embedding defined in Proposition \ref{embedd} (respectively, Proposition \ref{embedd-II}), for $k=1,2$. Then, the diagram:
\begin{equation}\label{eq:map-scalxcal}
\begin{CD}
\spec(R_2) @>i_2>> \scal(R_2) @>j_2>> \xcal(R_2)\\
@V{f^a}VV     @V{\scal(f)}VV      @V{\xcal(f^a)}VV\\
\spec(R_1) @>i_1>> \scal(R_1) @>j_1>> \xcal(R_1)
\end{CD}
\end{equation}
commutes.
\end{prop}
\begin{proof}
The left square of \eqref{eq:map-scalxcal} commutes by Proposition \ref{prop:scal-functor}(2).

Let now $\mathscr{Q}\in\scal(R_2)$. Then, using Proposition \ref{embedd-II}(1),
\begin{equation*}
j_1\circ\scal(f)(\mathscr{Q})=j_1(f^{-1}(\mathscr{Q}))=\{P \mid P\subseteq f^{-1}(\mathscr{Q})\},
\end{equation*}
while
\begin{equation*}
\begin{array}{rcl}
\xcal(f^a)\circ j_2(\mathscr{Q}) & = & \xcal(f^a)\left(\{P \mid P\subseteq\mathscr{Q}\}\right)\\
& = & \left(f^a\left(\{P \mid P\subseteq\mathscr{Q}\}\right)\right)^{\texttt{gen}}\\
& = & \left(\{f^{-1}(P) \mid P\subseteq\mathscr{Q}\}\right)^{\texttt{gen}}.
\end{array}
\end{equation*}

Let $Q\in\spec(R_1)$. If $Q\in\xcal(f^a)\circ j_2(\mathscr{Q})$, then 
$Q\subseteq f^{-1}(P)$ for some $P\subseteq\mathscr{Q}$; hence, $Q\subseteq f^{-1}(\mathscr{Q})$ and $Q\in j_1\circ\scal(f)(\mathscr{Q})$.

Conversely, suppose $Q\in j_1\circ\scal(f)(\mathscr{Q})$,   then $Q\subseteq f^{-1}(\mathscr{Q})$.   Therefore, $f(Q)\subseteq\mathscr{Q}$ and so  $f(Q)R_2\cap \Sigma_{\mathscr{Q}}=\emptyset$, where  $\Sigma_{\mathscr{Q}}:=R_2\setminus\mathscr{Q}$. 
It follows that $f(Q)R_2$ extends to a proper ideal of $\Sigma_{\mathscr{Q}}^{-1}R_2$, and in particular there is a prime ideal $P$ of $R_2$ such that $f(Q)\subseteq P$ and $\Sigma_{\mathscr{Q}}^{-1}P\neq \Sigma_{\mathscr{Q}}^{-1}R_2$.
Therefore,  $P\subseteq\mathscr{Q}$.
 It follows that $Q\subseteq f^{-1}(f(Q))\subseteq f^{-1}({P}) \ (\subseteq f^{-1}(\mathscr{Q}))$, and so $Q\in\xcal(f^a)\circ j_2(\mathscr{Q})$. Therefore, also the right square of \eqref{eq:map-scalxcal} commutes.
\end{proof}

It is obvious that, if $f$ is an isomorphism, $\scal(f)$ is   a homeomorphism. The converse does not hold;   for example, if $R_1\subset R_2$ is a proper integral extension of one-dimensional local domains, then $\scal(f)$ (like $f^a$ and $\xcal(f^a)$) is a homeomorphism, but $f$ is not an isomorphism.  More generally, we have:
 
\begin{cor}
Let $f:R_1\rightarrow R_2$ be a ring homomorphism, and let $f^a:\spec(R_2)\rightarrow \spec(R_1)$ be the associated spectral map. If $f^a$ is a topological embedding (respectively, a  homeomorphism) then so is $\scal(f)$.
\end{cor}
\begin{proof}
If $f^a$ is a topological embedding then, by    \cite[Proposition 4.4(1)]{FiFoSp-X(X)}, so is $\xcal(f^a)$, and thus also $\xcal(f^a)\circ j_2$ is a topological embedding. 
By Proposition \ref{prop:scalf}, it follows that $j_1\circ\scal(f)$ is a   topological embedding, and thus so is $\scal(f)$.

If $f^a$ is  a homeomorphism, then by the previous paragraph $\scal(f)$ is a topological embedding. Let $\mathscr{Q}\in\scal(R_1)$, and let $\mathscr{L}:=\bigcup\{\rad(f(P)R_2) \mid P\subseteq\mathcal{Q}\}$. Since $f^a$ is  a  homeomorphism, $\rad(f(P)R_2)$ is a prime ideal of $R_2$   (since the irreducible closed $\texttt{V}(P)$ subspace of  $\spec(R_1)$  is homeomorphic to $\texttt{V}(\rad(f(P)R_2))$ in $\spec(R_2)$), and so $\mathscr{L}$ is a prime semigroup.
We claim that $\scal(f)(\mathscr{L})=\mathscr{Q}$. 
Clearly if $q\in\mathscr{Q}$ then $f(q)\in\mathscr{L}$, and $q\in f^{-1}(\mathscr{L})=\scal(f)(\mathscr{L})$. Conversely, if $q\in\scal(f)(\mathscr{L})$, then $f(q)^n\in f(P)R_2$    for some $P\subseteq\mathscr{Q}$  and  for some $n\geq 1$.  Hence $q^n\in f^{-1}(f(P)R_2)=P$, the last equality coming from the bijectivity of $f^a$. Thus, $q\in P\subseteq\mathscr{Q}$. Therefore, $\scal(f)$ is surjective, and thus a homeomorphism.
\end{proof}

\begin{oss}\label{ex:ufd}
Despite the similarity between the properties enjoyed by $\xcal(R)$ and $\scal(R)$, there is however a  significant difference: while $\xcal(R)$ is a purely  topological construction (depending only on the topology of $\spec({R})$, see   \cite[Theorem 3.2 and Corollary 4.9]{FiFoSp-X(X)}), $\scal(R)$ depends also on the algebraic properties of $R$.   In particular, $\scal(R)$, in contrast with $\xcal(R)$ \cite[Theorem 4.5]{FiFoSp-X(X)} cannot be obtained from $\Spec(R)$ alone through a universal property. 
We provide now an   example of this fact, and another example  will be given later (Example \ref{ex:dedekind}).

Unlike in the case of $\xcal(R)$   \cite[Proposition 4.4]{FiFoSp-X(X)},  the image of $\spec(R)$ in $\scal(R)$ cannot be determined uniquely by topological means. 
For example, let $R$ be a unique factorization domain, and let ${\mathcal{P}}(R)$ be the set of   equivalence classes of  prime elements of $R$ 
modulo multiplication by units.   Any prime semigroup in $\scal(R)$ is uniquely determined by the prime elements that it contains,   and thus there is a bijective correspondence between $\scal(R)$ and the power set 
$\boldsymbol{\mathscr {B}} :=\boldsymbol{\mathscr{B}}(\mathcal{P}({R}))$ of ${\mathcal{P}}(R)$, which becomes  a homeomorphism if we take, as a  subbasis for $\boldsymbol{\mathscr{B}}$, the family of the subsets of $\boldsymbol{\mathscr{B}}$ of the form 
$\boldsymbol{V}({p}):=\{B\in \boldsymbol{\mathscr{B}} \mid p\notin B\}$, as $p$ runs in ${\mathcal{P}}(R)$. 
In particular, the topology of $\scal(R)$ depends uniquely on the cardinality of ${\mathcal{P}}(R)$, and thus it does not depend on other properties of $R$ or $\spec({R})$:  for example, it does not depend on the dimension of $R$.  Hence,  by cardinality reasons, there exists a homeomorphism
$\scal(\insZ)\simeq\scal(\insZ[X])$, but $j(\spec(\insZ))$ and $j(\spec(\insZ[X]))$ are not homeo\-morphic, and so  they do not correspond under any homeomorphism between $\scal(\insZ)$ and $\scal(\insZ[X])$.
\end{oss}

We prove next that the spectral space $\scal(R)$ is a retract of the spectral space $\xcal({R})$.

\begin{prop}\label{homeo}
Let $R$ be a ring, $j:  \scal({R})  \rightarrow   \xcal(R) $  the canonical embedding defined in  Proposition \ref{embedd-II}(1) and let  $\mathscr{P}:   \xcal(R)  \rightarrow \scal({R})$ be the map defined by setting $\mathscr{P}(Y) := \bigcup\{P \mid P\in Y\}$ for each $Y \in \xcal({R})$. Then:
\begin{enumerate}
\item[{\rm (1)}] $\mathscr{P}$ is   surjective and    spectral; 
\item[{\rm (2)}] $\mathscr{P}\circ j$ is the identity on $\scal(R)$; 
\item[{\rm (3)}] for every $Y\in\xcal(R)$, $(j\circ\mathscr{P})(Y)=\bigcap\{\mathtt{D}(a)\mid Y\subseteq \mathtt{D}(a)\}$.
\end{enumerate}
\end{prop}
\begin{proof}
(1) and (2). Let $\boldsymbol{U}(x)$ be a   basic  open set of $\scal({R})$, with $x \in R$. Then,
\begin{equation*}
\begin{array}{rcl}
\mathscr{P}^{-1}(\boldsymbol{U}(x)) & = & \{Y\in\xcal({R})\mid \mathscr{P}(Y)\in\boldsymbol{U}(x)\}=\\
& = & \{Y\in\xcal(R) \mid  x\notin\mathscr{P}(Y)\}\\
& = & \{Y\in\xcal(R) \mid  x\notin\bigcup\{P \mid P\in Y\}\}\\
& = & \{Y\in\xcal(R) \mid  x\notin P\text{~for every~}P\in Y\}=\\
& = & \{Y\in\xcal(R) \mid  Y\subseteq \texttt{D}(x)\}=\ucal(\texttt{D}(x))
\end{array}
\end{equation*}
which is a basic quasi-compact open set of $\xcal(R)$. Hence, $\mathscr{P}$ is   (continuous and) spectral.

The fact that $\mathscr{P}\circ j$ is the identity on $\scal(R)$ follows directly from Lemma \ref{semigroupprime} and Proposition \ref{embedd-II}(1), and in particular   it  implies that $\mathscr{P}$ is surjective.

(3) Let $Y\in\xcal(R)$. If $Y\subseteq\mathtt{D}(a)$, then $a\notin P$ for every $P\in Y$, and thus $a\notin\bigcup\{P \mid P\in Y\}=\mathscr{P}(Y)$. Hence,   if $Q\in (j\circ\mathscr{P})(Y)$ then   $a\notin Q$ and  so $Q\in\mathtt{D}(a)$. Conversely, suppose $Q$ belongs to the given intersection. If $Q\notin( j\circ\mathscr{P})(Y)$, then an element  $q\in Q\setminus\mathscr{P}(Y)$ would  exist. But this would imply $Y\subseteq\mathtt{D}(q)$ while $Q\notin\mathtt{D}(q)$, which is absurd.
\end{proof}

\begin{oss} \label{chiusinv}   As we observed at the beginning of the present section, we can define $\mathscr{P}(Y) :=\{P \mid P\in Y\}$ for each nonempty subset $Y$ of $\spec({R})$. In this case, we  can show  that  if  $Y_1,\ Y_2 \subseteq \spec({R})$  and if  $\chiusinv(Y_1)\subseteq \chiusinv(Y_2)$ then $\ms P(Y_1)\subseteq \ms P(Y_2)$. In particular, if $\chiusinv(Y_1) = \chiusinv(Y_2)$, then $\ms P(Y_1)=\ms P(Y_2)$, hence   $\ms P(Y)=\ms P(\chiusinv(Y))$ for each 
nonempty subset $Y$ of $\spec({R})$.

As a matter of fact, let $x\in R$ be such that $x\in \ms P(Y_1)\setminus\ms P(Y_2)$. Then $\texttt{D}(x)$  contains $Y_2$, and it is a closed set, with respect to the inverse topology of $\spec({R})$. Thus, by assumption, $\texttt{D}(x)\supseteq \chiusinv(Y_2) \supseteq \chiusinv(Y_1)\supseteq Y_1$. On the other hand, since  $x\in \ms P(Y_1)$, there   exist a prime ideal  $P \in Y_1$ such that $x\in P$,   and hence $Y_1 \not\subseteq \texttt{D}(x)$, which is a contradiction.
\end{oss}

In the next result, we characterize when the canonical embedding $\scal({R})\hookrightarrow \xcal({R})$ is a homeomorphism and, as a consequence, we deduce that, in general,  there are rings $R$ and inverse-closed subspaces $Y$ of $\spec({R})$ such that $ Y \subsetneq (j\circ\mathscr{P})(Y)$.

\begin{teor}\label{prop:scal-xcal}
Let $R$ be a ring. The following statements are equivalent.
\begin{itemize}
\item[(i)] The canonical embedding $j: \scal(R)\hookrightarrow \xcal(R)$ (defined in Proposition \ref{embedd-II}(1)) is   a homeomorphism.
\item[(ii)] The radical of every finitely generated ideal of $R$ is the radical of a principal ideal.
\item[(iii)] If $I$ is a finitely generated ideal of $R$ and $ I \subseteq \mathscr{Q} := \bigcup \{Q_\lambda \mid \lambda \in \Lambda\} \in \scal( R)$  (where $Q_\lambda \in \spec( R)$ for each $\lambda$), then $I \subseteq Q_\lambda$ for some $\lambda \in \Lambda$.
 \item[(iv)]  A basis for the open sets for the Zariski topology of $\xcal(R)$ is given by the collection $\{\ucal(\texttt{D}(x))\mid x\in R \}.$
\end{itemize}
\end{teor}
\begin{proof}
(i) $\Rightarrow$ (ii). By Proposition \ref{embedd-II}(1), $j$ is  a  homeomorphism if and only if it is surjective. 
Suppose   $j$ is  a homeomorphism,  and suppose there is a   nonzero  finitely generated ideal $I$ such that $\rad(I)\neq\rad(aR)$ for every $a\in R$. Consider $Y:=\texttt{D}(I)$: then, $Y$ is open and quasi-compact in the Zariski topology, and thus it is a closed set in the inverse topology. 
Since $j$ is surjective, there is a prime semigroup $\mathscr{Q}$ such that $Y=j(\mathscr{Q})$. Set $\Sigma_{\mathscr{Q}}:=R\setminus\mathscr{Q}$ and $R_{\mathscr{Q}} :=\Sigma_{\mathscr{Q}}^{-1}R$. 
 Suppose $I\subseteq\mathscr{Q}$: then, $IR_{\mathscr{Q}} \neq R_{\mathscr{Q}}$, so that there is a prime ideal $P$ such that $I\subseteq P$ and $PR_{\mathscr{Q}} \neq R_{\mathscr{Q}} $. 
 Therefore, $P\notin \texttt{D}(I)=Y$.    On the other hand,  $P\subseteq\mathscr{Q}$, and thus $P\in j(\mathscr{Q})=Y$: a contradiction. 
 Henceforth, $I\not \subseteq\mathscr{Q}$, i.e., there exists an element  $s\in I\cap \Sigma_{\mathscr{Q}}$. However, since the radical of the ideal $sR$ cannot be equal to $\rad(I)$, and $sR\subseteq I$, there is a prime ideal $Q$ containing $sR$ but not $I$.
 Hence, $Q\in Y$, while $QR_{\mathscr{Q}}=R_{\mathscr{Q}}$, and thus $Q\notin j(\mathscr{Q})$. Again, this conflicts with the assumptions, and so we conclude that $Y$ is not in the image of $j$.

(ii) $\Rightarrow$ (iii). Let  $I$ be a   nonzero finitely generated ideal of $R$ and assume that $ I \subseteq \mathscr{Q}$.
By hypothesis, $\rad(I)=\rad(sR)$ for some $s$, and we can suppose $s\in I$. Since $I\subseteq \bigcup \{Q_\lambda \mid \lambda \in \Lambda\}$, then $s\in Q_\lambda$ for some $\lambda\in\Lambda$ and, hence, $I\subseteq\rad(I)=\rad(sR)\subseteq Q_\lambda$.

(iii) $\Rightarrow$ (i). Let $Y\in\xcal(R)$, and let $\mathscr{P}(Y) = \bigcup \{Q_\lambda \mid \lambda \in \Lambda\} \in \scal( R)$. We claim that $j(\mathscr{P}(Y))=Y$.  Clearly, $Y \subseteq j(\mathscr{P}(Y))$ (Proposition \ref{homeo}(3)). On the other hand, suppose $P\in j(\mathscr{P}(Y))\setminus Y$. Then, since $Y=\chiusinv(Y)$, there is a basic closed set $\Omega=\texttt{D}(I)$ of the inverse topology on $\spec({R})$,
such that $Y\subseteq\Omega$ but $P\notin\Omega$. 
Since $\Omega$ is quasi-compact in the Zariski topology, we can suppose $I$ finitely generated. 
The fact that $P\notin \texttt{D}(I)$ implies $I\subseteq P$. 
On the other hand,  $P\in j(\mathscr{P}(Y))$, hence  $P\subseteq\bigcup\{Q\mid Q\in Y\}$.
Therefore, by hypothesis, there is a $\overline{Q}\in Y \ (\subseteq \Omega)$ such that $I\subseteq\overline{Q}$; but this would imply $\overline{Q}\notin \texttt{D}(I)$, which is absurd. Hence, $Y$ is in the image of $j$, and so $j$ is surjective.

 Clearly, (ii)$\Rightarrow$(iv) since a basis for the open sets of $\xcal(R)$  is given by $ \ucal(\texttt{D}(J))$ for $J$ varying among the finitely generated ideals of $R$.
Conversely, let $J$ be a   nonzero  finitely generated ideal of $R$. Since $\texttt{D}(J)\in \ucal(\texttt{D}(J))$, by assumption there is an element $x\in R$ such that $\texttt{D}(J)\in \ucal(\texttt{D}(x))\subseteq \ucal(\texttt{D}(J))$, that is, $\texttt{D}(x)=\texttt{D}(J)$ and, in other words, $\rad(xR)=\rad(J)$. 
\end{proof}

 An example where the previous theorem can be applied is when $R$ contains an uncountable field but its spectrum is only countable \cite[Proposition 2.5]{sharp-vamos}.

In case $\spec({R})$ is a Noetherian space, we have the following.

\begin{cor}\label{j-homeo}
Let $R$ be a ring. The following statements are equivalent.
 \begin{itemize}
 \item[(i)]
 The canonical embedding $j: \scal(R)\rightarrow \xcal(R)$  is {\bc a} homeomorphism and  $\spec( R)$ is a Noetherian space.
 \item[(ii)] Every prime ideal of $R$ is the radical of a principal ideal.
  \item[(iii)]  If $I$ is an ideal of $R$ and $ I \subseteq \mathscr{Q} := \bigcup \{Q_\lambda \mid \lambda \in \Lambda\} \in \scal( R)$  (where $Q_\lambda \in \spec( R)$ for each $\lambda$), then $I \subseteq Q_\lambda$ for some $\lambda \in \Lambda$.
   
   \item[(iv)]  If $P$ is a prime ideal of $R$ and $P \subseteq \mathscr{Q} := \bigcup \{Q_\lambda \mid \lambda \in \Lambda\} \in \scal( R)$  (where $Q_\lambda \in \spec( R)$ for each $\lambda$), then $P \subseteq Q_\lambda$ for some $\lambda \in \Lambda$.
   \end{itemize}
\end{cor}
\begin{proof}
The equivalence of (i) and (ii) follows from the previous theorem, since $\spec(R)$ is Noetherian if and only if every radical ideal is the radical of a finitely generated ideal  (see for instance \cite{ohpe} or \cite[Theorem 3.1.11]{FoHuPa}). The equivalences (ii) $\Leftrightarrow$ (iii) $\Leftrightarrow$ (iv) are due to W.W. Smith \cite{sm}.
\end{proof}

\begin{oss}
Rings verifying property (iii) of the previous corollary has been called \emph{compactly packed} in \cite{revi}.
\end{oss}

\begin{oss}
It is well known that the rings verifying the equivalent  conditions (ii)-(iv) of the previous corollary have Noetherian spectrum. On the other hand, Theorem \ref{prop:scal-xcal} implies that $j$ is surjective for any B\'ezout domain and there are examples of  B\'ezout domains (or, even, valuation domains) $R$ such that $\spec(R)$ is not Noetherian. Therefore,  for an arbitrary ring $R$, conditions (ii) $\Leftrightarrow$ (iii) $\Leftrightarrow$ (iv) of the previous corollary   do not provide a characterization of when $j: \scal(R)\longrightarrow \xcal(R)$  is a homeomorphism. 
In other words, 
the property that $j: \scal(R)\rightarrow \xcal(R)$ is a homeomorphism does not depend only on the topology of the spectrum of $R$, but also on the algebraic structure of $R$.
\end{oss}

\begin{oss} 
(a) Let $R$ be a ring. If $\boldsymbol{\mathcal{T}}:=\{\mathscr{Q}_\alpha \mid \alpha\in A\}$ is a nonempty subset of $\scal(R)$, then $\bigcup\{\mathscr{Q}_\alpha\mid \alpha\in A\}$ is a prime semigroup of $R$, and it is easily seen that it is the supremum of $\boldsymbol{\mathcal{T}}$ in $\scal(R)$, with the order induced by the hull-kernel topology,  that is the set theoretic inclusion.

For investigating  the existence of the infimum of $\boldsymbol{\mathcal{T}}$,   we cannot argue in a dual way, since the natural candidate $\bigcap\{\mathscr{Q}_\alpha \mid \alpha\in A\}$ is not, in general, a prime semigroup (for example, if $P$ and $Q$ are incomparable prime ideals of $R$, they are both prime semigroups, but $P\cap Q$ is not).  However,  we can show that an infimum can be determined in many cases. More precisely,
let $j: \scal({R}) \hookrightarrow \xcal({R})$    be the topological embedding defined in Proposition \ref{embedd-II}(1). Then, the set $j(\boldsymbol{\mathcal{T}}):=\{j(\mathscr{Q}_\alpha)\mid \alpha\in A\}$ 
is a family of closed subsets in the inverse topology of $\spec(R)$, 
and so {\sl if} $C_{\boldsymbol{\mathcal{T}}}:=\bigcap\{j(\mathscr{Q}_\alpha) \mid \alpha\in A\}$ is {\sl nonempty}  (for instance, since $X^{\mbox{\tiny\texttt{inv}}}$ is compact, for this assumption it suffices that the set of the $j(\mathscr{Q}_\alpha)$ satisfies the finite intersection property), then   it 
still belongs to $\xcal(R)$, and it is the infimum of $j(\boldsymbol{\mathcal{T}})$ in $\xcal({R})$.
 We claim that $C_{\boldsymbol{\mathcal{T}}}=j(\mathscr{Q}_0)$ for some $\mathscr{Q}_0\in\scal({R})$.
 More precisely, we claim that $\mathscr{Q}_0=\bigcup\{Q \mid Q \in C_{\boldsymbol{\mathcal{T}}}\}$. 
 
 Indeed, if $Q\in C_{\boldsymbol{\mathcal{T}}} $ then $Q \in j(\mathscr{Q}_0)$ by Proposition \ref{embedd-II}(1).
 Conversely, if $P\in j(\mathscr{Q}_0)$, then $P\subseteq\mathscr{Q}_0\subseteq\mathscr{Q}_\alpha$ for every $\alpha\in A$, and thus (again, by Proposition \ref{embedd-II}(1)) $P\in j(\mathscr{Q}_\alpha)$ for every $\alpha$, i.e., $P\in C_{\boldsymbol{\mathcal{T}}}$.
 Therefore, $j(\mathscr{Q}_0)$  is the infimum of $j(\boldsymbol{\mathcal{T}})$ in $j(\scal(R))$, and since $j$ is a homeomorphism between $\scal(R)$ and its image in $\xcal({R})$,  it follows that $\mathscr{Q}_0$ is the infimum of $\boldsymbol{\mathcal{T}}$ in $\scal({R})$.

 (b)  From (a), it follow by construction that the topological embedding $j: \scal({R}) \hookrightarrow \xcal({R})$
preserves the infimum, in the cases where it exists.
However, the embedding $j$ in general does not preserve the supremum.

For example, let $D$ be a local unique factorization domain of dimension 2, and let $Y_1,Y_2$ be two nonempty disjoint sets of prime ideals such that $Y_1\cup Y_2$ is the set $\spec^1(D)$ of  prime ideals  of height 1 of $D$. 
If $\mathscr{Q}_i:=\bigcup\{P \mid P\in Y_i\}$, then 
$j(\mathscr{Q}_i) = Y_i \cup \{(0)\}$, and thus $j(\mathscr{Q}_1)\cup j(\mathscr{Q}_2)=\spec^1(D)\cup \{(0)\} \subsetneq \spec(D)$.
 However,
 $\mathscr{Q}_1\cup\mathscr{Q}_2$ is equal to the set of non-units of $D$, so that $j(\mathscr{Q}_1\cup\mathscr{Q}_2)=\spec(D)$. 
 
 On the other hand, if $\{P_1, P_2,\ldots,P_n\}$ is a finite set of prime ideals (and thus, in particular, of prime semigroups) of $R$, then $j(P_1\cup P_2\cup \cdots\cup P_n)=j(P_1)\cup j(P_2)\cup \cdots\cup j(P_n)$. Indeed, by Proposition \ref{embedd-II}(1), $Q\in j(P_1\cup P_2\cup \cdots\cup P_n)$ if and only if $Q\subseteq P_1\cup P_2\cup\cdots\cup P_n$ and, by prime avoidance, this is equivalent to $Q\subseteq P_i$ for some $i$, and thus to $Q\in j(P_i)$ for some $i$.
\end{oss}

\section{The integral domain case} 

Let $D$ be an integral domain, and recall that the set $\overr(D)$ of the overrings of $R$ has a natural topological structure (see Section \ref{sect:overrings}). Then, there is a natural map
\begin{equation*}
\begin{aligned}
\ell_0\colon\Spec(D) & \longrightarrow\overr(D) \\
P & \longmapsto D_P,
\end{aligned}
\end{equation*}
which is a topological embedding \cite[Lemma 2.4]{dobbs-fedder-fontana}. We show   next that $\scal(D)$ admits a similar interpretation   with respect to $\overr(D)$.

\begin{prop}\label{prop:mappe-domains}
Let  $D$ be an integral  domain with quotient field $K$  and let   {\rm${\overr}(D)$} be the set of the overrings of $D$, endowed with the Zariski topology.
 
\begin{enumerate}
 \item[{\rm (1)}]    Let $\mathscr{Q}\in  \scal(D) $ and set as above $\Sigma_{\mathscr{Q}}:=D\setminus\mathscr{Q}$ and $D_{\mathscr{Q}} :=\Sigma_{\mathscr{Q}}^{-1}D$. 
The map

\begin{equation*}
\begin{aligned}
\ell\colon \scal(D)& \longrightarrow\overr(D) \\
\mathscr{Q} & \longmapsto   D_{\mathscr{Q}}
\end{aligned}
\end{equation*}
is a topological embedding that extends the map $\ell_0$ defined above.

\item [{\rm (2)}]The map 

\begin{equation*}
\begin{aligned}
\omega\colon\overr(D) & \longrightarrow\xcal(D) \\
T & \longmapsto \qspec^{\widetilde{\, \wedge_{\{T\}}}}(D)
\end{aligned}
\end{equation*}
is  a continuous map of spectral spaces. 
Moreover,   if $T \in \overr(D)$ and the canonical embedding $\tau:D\longrightarrow T$ is flat, then   $\omega(T) =  \tau^a(\spec(T))$.

\item[{\rm (3)}]  The composition $\omega \circ \ell: \scal(D) \hookrightarrow \xcal(D)$ coincides with the topological embedding $j$ defined in Proposition \ref{embedd-II}(1).
\end{enumerate}
\end{prop}
\begin{proof}
(1) Since $\{\texttt{B}(x)\mid  x\in K\}$ is a subbasis of open sets for $ \overr(D)$, to get continuity of $\ell$ it suffices to prove that, if $x\in K$, then $\ell^{-1}(\texttt{B}(x))$ is open in $\scal(D)$. 
Take a semigroup prime $\ms Q\in \ell^{-1}(\texttt{B}(x))$, and let 
$d,s\in D$ with $s\notin \ms Q$ such that $x=\frac{d}{s}\in D_{\ms Q}$. Then, we have $\ms Q\in \boldsymbol{U}(s)\subseteq \ell^{-1}(\texttt{B}(s^{-1}))\subseteq \ell^{-1}(\texttt{B}(x))$, that is, $\ell^{-1}(\texttt{B}(x))$ is open in $\scal(D)$.

To prove that $\ell$ is a topological embedding  it is now sufficient to note that, for any nonzero element $f \in D$, we have $\ell(\boldsymbol{U}(f))=\ell(\scal(D))\cap \texttt{B}(f^{-1})$. 
The inclusion $\subseteq $ is trivial. Conversely, if $T\in \ell(\scal(D))\cap \texttt{B}(f^{-1})$,
 then there are a semigroup prime $\ms Q$ and elements $d,s\in D$ such that $s\notin \ms Q$ and $\frac{1}{f}=\frac{d}{s}\in D_{\ms Q}=T$.
  It follows $s=df\notin \ms Q$ and, a fortiori, by definition of semigroup prime, $f\notin \ms Q$. Then $T\in \ell(\boldsymbol{U}(f))$, and thus the proof is complete.

(2)  It is sufficient to note that $\omega$ is the composition of three continuous maps,  namely the topological embedding $  \iota: \overr(D) \hookrightarrow \inssemistar(D)$
(defined,  for each overring $T$ of $D$, by   $ \iota(T) :=  \wedge_{\{T\}} $ \cite[Proposition 2.5]{FiSp}), the   continuous surjection  $\widetilde{\Phi}: \inssemistar(D) \twoheadrightarrow
\inssemistabft(D)$ (defined,  for each $ \star \in \inssemistar(D)$, by $  \widetilde{\Phi}(\star) :=  \widetilde{\star}$   \cite[Proposition 4.3(2)]{FiFoSp-JPAA}), and the homeomorphism
$ \Delta: \inssemistabft(D)   \xrightarrow{\sim}  \xcal(D)$ (defined,  for each  $ \star \in \inssemistabft(D)$, by $ \Delta(\star) :=  \qspec^\star(D)$  \cite[Proposition 5.2(1)]{FiFoSp-X(X)}).

Suppose $T$ is flat over $D$. In order to show that $\qspec^{\widetilde{\, \wedge_{\{T\}}}}(D) = \{ Q\cap D \mid Q \in \spec(T) \}$  we observe  that,  even if $T$ is not $D$-flat, the equality  $ \qspec^{\widetilde{\, \wedge_{\{T\}}}}(D)=\chiusinv(\qspec^{\wedge_{\{T\}}}(D))$ holds, in view of Remark \ref{qspec-inv}, since $\wedge_{\{T\}}$ is a semistar operation of finite type.
Moreover,  $P \in \qspec^{\wedge_{\{T\}}}(D)$ if and only if $P = PT \cap D$.  Hence, $ \tau^a(\spec(T)) = \{ Q\cap D \mid Q \in \spec(T) \} \subseteq \qspec^{\wedge_{\{T\}}}(D) $.  Conversely,   assuming that $T$ is $D$-flat, if  $P \in \qspec^{\wedge_{\{T\}}}(D)$ and if $Q \in \spec(T)$ and it is minimal over $PT$ then, by flatness, $Q \cap D =P$ \cite[Section 1-6, Exercise 37]{kap} and so $P \in  \tau^a(\spec(T))$. 

(3) is a straightforward consequence of the definitions.
\end{proof}

When we specialize  our investigation to the class of Pr\"ufer domains,   we obtain more precise statements.

\begin{prop}\label{prufer-case}
Let $D$ be a Pr\"ufer domain.  Then, the chain of canonical maps
\begin{equation*}
\overr(D) \xlongrightarrow{\iota} \insfinss(D) \xlongrightarrow{\widetilde{\Phi}} \inssemistabft(D)  \xlongrightarrow{\Delta} {\boldsymbol{\mathcal{X}}}(D)
\end{equation*}
is a chain of homeomorphisms, and $\widetilde{\Phi}$ is the identity. Moreover, the composition $\Delta\circ\widetilde{\Phi}\circ\iota$ coincides with the map $\omega$ defined in Proposition \ref{prop:mappe-domains}(2), and $\omega(T):=\qspec^{\wedge_{\{T\}}}(D)$ for all $T\in\overr(D)$.
\end{prop}
\begin{proof}
 The map $\Delta: \inssemistabft(D)\rightarrow  \boldsymbol{\mathcal{X}}(D)$ (defined by $\Delta(\star) := \qspec^\star(D)$ for each $\star$ spectral semistar operation of finite type on $D$) is a homeomorphism
by   \cite[Proposition 5.2(1)]{FiFoSp-X(X)}.

Since $D$ is a Pr\"ufer domain, each of its  overrings is $D$-flat \cite[Theorem 1.1.1]{FoHuPa}. Then, the canonical map
 ${\widetilde{\Phi}} \circ \iota :  \overr(D) \longrightarrow \inssemistabft(D)$, 
 $T \mapsto \wedge_{\{T\}} = \widetilde{\,\wedge_{\{T\}}}$, is a topological embedding  (proof of Proposition \ref{prop:mappe-domains}(2)  or  \cite[Proposition 2.5]{FiSp}).

We need to show that $\insfinss(D)=\inssemistabft(D)$. Indeed, if $\star\in\insfinss(D)$, then the domain $D^\star$, as an overring of a Pr\"ufer domain, is still  a Pr\"ufer domain. Hence, $  \wedge_{\{D^\star\}} = \widetilde{\, \wedge_{\{D^\star\}}\,}$,   since $D^\star$ is $D$-flat, and $D^\star$ admits a unique star operation of finite type. It follows that $\star|_{\F(D^\star)}:
  \F(D^\star) \rightarrow \F(D^\star)$ is the identity star operation of $D^\star$. On the other hand note that, for each $F \in \f(D)$,
\begin{equation*}
F^\star=(FD)^\star=(FD^\star)^\star=FD^\star\,.
\end{equation*} 
Therefore, we have $\star =  \wedge_{\{D^\star\}}$ and so $\iota$ is surjective.

 The equality $\omega = \Delta\circ\widetilde{\Phi}\circ\iota$ holds in general (see the proof of Proposition \ref{prop:mappe-domains}(2)) 
and the last  claim follows from the fact that $\wedge_{\{T\}} = \widetilde{\,\wedge_{\{T\}}}$, since every overring $T$ of the Pr\"ufer domain $D$ is $D$-flat.
\end{proof}

Recall that  an integral domain  $D$ is a {\it QR-domain} if each overring of $D$ is a ring of fractions of $D$ (for more details see, for example,  \cite{gioh} and \cite{he}). For example, a B\'ezout domain is a QR-domain \cite[page 250 Exercise 10(b)]{gi}.

\begin{cor}\label{prop:QR}
Let $D$ be a QR-domain. Then, the chain of maps
\begin{equation*}
\boldsymbol{\mathcal{S}}(D)  \xlongrightarrow{\ell}  \overr(D)  \xlongrightarrow{\omega}\boldsymbol{\mathcal{X}}(D)
\end{equation*}
 is a chain of homeomorphisms.
\end{cor}
\begin{proof}
 By Proposition \ref{embedd-II}(3), $\ell$ is a topological embedding, and the hypothesis that $D$ is a QR-domain guarantees that $\ell$ is also surjective. Therefore, $\ell$ is a homeomorphism.   Since a QR-domain is -- in particular -- a Pr\"ufer domain   \cite[p. 334]{gi}, then we know from Proposition \ref{prufer-case} that $\omega$ is a homeomorphism.  The claim follows.
\end{proof}

\begin{ex}\label{ex:dedekind}
Consider a Dedekind domain $D$ such that the class group $\Cl(D)$ of $D$ is not a torsion group (an explicit example is   given  by $D:=K[X,Y]/(X^2-Y^3+Y+1)$, where $K$ is an algebraically closed field; see   \cite[Sections 3 and 4]{gioh} and \cite[page 146]{re}, and for a general result \cite[Theorem 14.10]{fossum}). Then, there is a maximal ideal $P$ of $D$ such that the class $[P]$ has infinite order in $\Cl(D)$, i.e., $P^n$ is never principal or, equivalently, no principal ideal is $P$-primary \cite[Proposition 6.8]{fossum}. 
Let $Y:=\spec(D)\setminus\{P\}$: then, $Y$ is closed in the inverse topology,   since it is a quasi-compact open subspace of $\spec(D)$, endowed with the Zariski topology.
 We claim that $Y\notin j(\scal(D))$. If it was, say $Y=j(\mathscr{Q})$, then $\mathscr{Q} \in \scal(D)$ must contain every element of $Y$, but there must be an $x\in P$ such that $x\notin\mathscr{Q}$. However, the ideal $xD$ is not $P$-primary, and so there  also  exists a prime ideal $Q$ of $D$, $Q\neq P$, such that $x\in Q$. This contradicts $Y=j(\mathscr{Q})$, and so $j$ is not surjective.

On the other hand, if $D'$ is a principal ideal domain, then  $j':\scal(D')\rightarrow\xcal(D')$ is surjective (Corollary \ref{prop:QR}). Moreover, we can always   find  a principal ideal domain $D'$ such that the cardinality of $\Max(D')$ is equal to the cardinality of $\Max(D)$ (it suffices to take $D':=F[ T]$, where $F$ is a field with the same cardinality of $\Max(D)$)  and $T$ is an indeterminate over $F$. Then, $\spec(D')$ and $\spec(D)$ are homeomorphic (it is enough to take any bijection between $\Max(D')$ and $\Max(D)$ then extend it to a bijection $\rho:\spec(D')\rightarrow\spec(D)$ such that $\rho((0))=(0)$), but $j'$ is surjective while $j$ is not.
\end{ex}

\begin{oss}\label{rk-dedekind}
  Note that, by \cite[Theorem 2.2]{revi}, when $R:=D$ is a Dedekind domain,  the condition that the canonical map $\boldsymbol{\mathcal{S}}(D) \rightarrow \xcal(D)$ is a homeomorphism and, hence, the equivalent conditions of Corollary \ref{j-homeo} are equivalent to the following:
 \begin{itemize}
 \item[(iv)] The ideal class group of $D$ is torsion.
  \item[(v)] $D$ is a QR-domain.
 \end{itemize}
 \end{oss}

\section{The space of semigroup primes of the Nagata ring}

 Our next goal is to show that, for each ring $R$, the spectral space $\xcal(R)$ can be embedded in a space of prime semigroups   of a different ring $A$: more precisely, we will show that we can choose $A$ to be the Nagata ring of $R$.
   
  \smallskip
   
Recall that, given a ring $R$ and an indeterminate $T$ over $R$, the {\it Nagata ring}  $R(T)$  of $R$   is the localization $S^{-1}R[T] $, where $S$ is the multiplicative set of all the primitive polynomials of $R[T]$.
It is well known \cite[Proposition 33.1(1)]{gi} that $S=R[T]\setminus \bigcup\{M[T]\mid M\in \Max({R}) \}$.  
Let $g: R \hookrightarrow R(T)$ be the canonical embedding. For the sake of simplicity, we identify $R$ with $g({R})$ inside $R(X)$. It is clear that the spectral map $g^a: \spec(R(T)) \rightarrow \spec({R})$ is surjective.  For uses of Nagata rings and related rings of rational functions in the context of star and semistar operations, see \cite{gi}, \cite{fo-lo-2003}, \cite{folo-2006}, \cite{ch-06}, \cite{ch-08}, \cite{ch-10}, \cite{ch-ka}, \cite{do-sa}, \cite{HK-03},  \cite{ka-89}, \cite{ja} and \cite{ok-ma-97}.
 
Now, we consider  another map $\gamma:\spec({R})\rightarrow \spec(R(T))$ by setting $\gamma({P}) := PR(T)$ for each $P \in \spec({R})$: this map is well-defined and injective (since $IR(T)\cap R=I$, for all ideals $I$ of $R$ \cite[Proposition 33.1(4)]{gi}).   Clearly, $\gamma \circ g^a$ is the identity map of $\spec(R)$. Further properties are given next. 

\begin{lemma} \label{gamma} Let  $\gamma:\spec({R})\rightarrow \spec(R(T))$ and $g^a: \spec(R(T)) \rightarrow \spec({R})$ be as above.
\begin{enumerate}
\item[\rm (1)] The map $\gamma$ is a spectral  embedding and  $g^a$ is a spectral retraction. 
\end{enumerate}
 Let $Y$ and $Z$ two nonempty subsets of $\spec({R})$, and, for any $X\subseteq\Spec(R)$, let $\boldsymbol{\ms Q}(X):= \bigcup\{PR(T)\mid P \in X \}\subseteq R(T)$.
\begin{enumerate}
\item[\rm (2)]  If $\chiusinv(Y)=\chiusinv(Z)$, then also $\chiusinv(\gamma(Y))=\chiusinv(\gamma(Z))$. 

\item[\rm (3)]   The equality $\boldsymbol{\ms Q}(Y)=\boldsymbol{\ms Q}(Z)$ holds if and only if $\chiusinv(Y)=\chiusinv(Z)$. 
\end{enumerate}
\end{lemma}
\begin{proof}
(1) Take a  nonzero  element $f/p\in R(T)$, where $f, p\in R[T]$ and $p$ is primitive, and write $f: =  a_0+a_1T+\ldots + a_nT^n$. For any prime ideal $P$ of $R$, we have:
$$
\frac{f}{p}\notin PR(T) \iff  f\notin PR[T]\iff    P\nsupseteq   (a_0, a_1\ldots,a_n)R,
$$
that is, $\gamma^{-1}\left(\texttt{D}\left(\frac{f}{p}R(T)\right)\right)=\texttt{D}((a_0,a_1, \dots, a_n)R)$. 
This proves that $\gamma$ is   (continuous and) spectral. Moreover, for each $x\in R$ we have $\gamma(\texttt{D}(xR))=\texttt{D}(xR(T))\cap {\rm Im}(\gamma)$, and thus $\gamma$ is a topological embedding.   The conclusion follows from the fact that $g^a \circ \gamma$ is the identity of $\spec({R})$.

(2) Assume that $\chiusinv(Y)=\chiusinv(Z)$. By definition, a basis for closed sets for the inverse topology of $\spec(R(T))$ is given by the quasi-compact open subspaces of $\spec(R(T))$ (when endowed with the Zariski topology). Thus, we have to prove that, for any nonzero finitely generated ideal $J$ of $R(T)$, we have $\gamma(Y)\subseteq \texttt{D}(J)$ if and only if $\gamma(Z)\subseteq \texttt{D}(J)$.
 Let $\frac{f_1}{p_1}, \frac{f_2}{p_2},\ldots, \frac{f_r}{p_r}\in R(T)$ be generators of the ideal $J$,   where $f_i, p_i\in R[T]$ and $p_i$ is primitive, for $i=1, 2, \ldots,r$,  and  let $C_i$ be the content of $f_i$. 
 Then $\gamma(Y)\subseteq \texttt{D}(J)$ if and only if for any $P\in Y$ there is some index $i$ such that $\frac{f_i}{p_i}\notin PR(T)$,
  that is $f_i\notin PR[T]= PR(T)\cap R[T]$. 
  In other words,  $ P \nsupseteq C_i$, i.e., $P \in \texttt{D}(C_i)$. 
   If we set $C :=C_1+ C_2+ \ldots+C_r$, the previous argument shows that $\gamma(Y)\subseteq \texttt{D}(J)$ if and only if $Y\subseteq \texttt{D}({C})$.
   Since $C$ is a finitely generated ideal of $R$, the set $D({C})$ is a quasi-compact open subspace of $\spec({R})$,  and thus also $Z\subseteq \texttt{D}({C})$, because  $Y$ and $Z$ have the same closure, with respect to the inverse topology. 
   Thus, any prime ideal $Q$ of $Z$ does not contain some coefficient of some polynomial $f_i$, and then $\frac{f_i}{p_i}\notin QR(T)$, that is $\gamma(Z)\subseteq \texttt{D}(J)$.  

(3) If $\chiusinv(Y)=\chiusinv(Z)$ then $\chiusinv(\gamma(Y))=\chiusinv(\gamma(Z))$, by part (2). Thus, the equality $\boldsymbol{\ms Q}(Y)=\boldsymbol{\ms Q}(Z)$ holds by Remark \ref{chiusinv}. Conversely, assume that $\boldsymbol{\ms Q}(Y)=\boldsymbol{\ms Q}(Z)$, and let $ J:= (a_0, a_1, \ldots,a_n)R$ be a nonzero finitely generated ideal of $R$. We have to prove that 
$Y\subseteq \texttt{D}(J)$ if and only if $Z\subseteq \texttt{D}(J)$. Suppose that $Y\subseteq \texttt{D}(J)$. Then, if $f:=  a_0+a_1T+\ldots +a_nT^n  \in R[T]$, we have $f\notin PR[T]= PR(T)\cap R[T]$, for each $P\in Y$, that is $\frac{f}{1}\notin \boldsymbol{\ms Q}(Y)= \boldsymbol{\ms Q}(Z)$.
 In other words, $f\notin QR[T]$, for each $Q\in Z$, i.e., $Z\subseteq \texttt{D}(J)$. 
\end{proof}

Now, we are in condition to prove that the spectral space $\xcal({R})$ can be embedded in the spectral space of prime semigroups of the Nagata ring $R(T)$.

\begin{prop}\label{nagata}
Let $R$ be a ring, $j:\scal({R}) \hookrightarrow \xcal({R})$  the spectral embedding defined in   Proposition \ref{embedd-II}(1),   $g: R \hookrightarrow R(T)$  the canonical ring embedding and let $\scal(g): \scal(R(T)) \rightarrow \scal({R})$   be the spectral map associated to $g$ defined in \eqref{eq:scalf}.   Define $\boldsymbol{\nu}$ as the map
\begin{equation*}
\begin{aligned}
\boldsymbol{\nu}\colon\xcal(R) & \longrightarrow\scal(R(T)) \\
Y & \longmapsto \bigcup\{PR(T) \mid P \in Y \}.
\end{aligned}
\end{equation*}
The following properties hold.
\begin{enumerate}
\item[{\rm (1)}] $\boldsymbol{\nu}$ is a spectral embedding.
\item[{\rm (2)}]   $\scal(g) \circ \boldsymbol{\nu} \circ j $ is the identity of $\scal({R})$.  In particular, 
$\scal(g): \scal(R(T)) \rightarrow \scal({R})$ is a topological retraction.
\item[{\rm (3)}]  If $\mathscr{P}  : \xcal(R)  \rightarrow \scal({R}) $ is the map defined in Proposition \ref{homeo}, then $\mathscr{P}=\scal(g)\circ\boldsymbol{\nu}$.
\end{enumerate}
\end{prop}
\begin{proof} (1).
By Lemma \ref{gamma}(3), the map $\boldsymbol{\nu}$ is injective. 
Now, let $  0 \neq \frac{f}{p}\in R(T)$, where $f, p\in R[T]$ and $p$ is primitive and let $C$ be the content of the polynomial $f$. Then,  using the notation of Lemma \ref{gamma}(3), \begin{equation*}
\begin{array}{rl}
\boldsymbol{\nu}^{-1}\left(\boldsymbol{U}\left(\frac{f}{p}\right)\right)= & \hskip -6pt \{Y\in \xcal({R}) \mid \frac{f}{p}\notin \boldsymbol{\ms Q}(Y) \} \\
= &\hskip -6pt  \{Y\in \xcal({R}) \mid f\notin PR[T] \mbox{ for all } P \in Y \}=\boldsymbol{\mathcal{U}}(\texttt{D}({C})).
\end{array}
\end{equation*}
This proves that $\boldsymbol{\nu}$ is   continuous and  spectral. On the other hand, with similar arguments, it can be shown that, given 
$  a_0, a_1, \ldots, a_n \in R$, if $f:=  a_0+a_1T+\ldots+a_nT^n  \in R[T]$ we have 
$$
\boldsymbol{\nu}\left(\texttt{D}(a_0, a_1, \ldots a_n)\right)={\rm Im}(\boldsymbol{\nu})\cap \boldsymbol{U}\left(\frac{f}{1}\right),
$$
that is, $\boldsymbol{\nu}$ is a topological embedding.

(2)    Let $ \mathscr{P} \in \scal({R})$. Let $Y$ be a nonempty set of prime ideals of $R$ such that $ \mathscr{P} = \mathscr{P}(Y) =\bigcup \{P \in \spec({R})\mid P\in Y\}$  (Lemma \ref{semigroupprime}). 
Set $\boldsymbol{\ms Q}(Y) := \bigcup \{PR(T)\mid P\in Y\} \in \scal(R(T))$.
Recall that, for each prime ideal $P\in \spec({R})$, $PR(T) \cap R =g^{-1}(PR(T))=P$ \cite[Proposition 33.1(4)]{gi}.  
Then, 
\begin{equation*}
\begin{array}{rl}
\scal(g) \circ \boldsymbol{\nu} \circ j(\mathscr{P})  
= &\hskip -6pt  \scal(g)(\boldsymbol{\ms Q}(Y)) = g^{-1}(\boldsymbol{\ms Q}(Y)) \\
= &\hskip -6pt  \left(\bigcup \{PR(T)\mid P\in Y\} \right) \cap R \\ 
 =&  \hskip -6pt  \bigcup \left(\{PR(T) \mid P\in Y\} \cap R\right) \\
= &\hskip -6pt\bigcup \{P \in \spec(R)\mid P\in Y\} = \mathscr{P}\,.
\end{array}
\end{equation*}

(3) Let $Y\in\xcal(R)$. Then, we have
\begin{equation*}
(\scal(g)\circ\boldsymbol{\nu})(Y)=g^{-1}(\boldsymbol{\nu}(Y))=g^{-1}\left(\bigcup_{P\in Y}PR(T)\right)= \bigcup_{P\in Y}g^{-1}(PR(T)).
\end{equation*}
However,  as noted above, $g^{-1}(PR(T))=P$ for every $P\in\Spec(R)$, and thus $(\scal(g)\circ\boldsymbol{\nu})(Y)=\bigcup\{P\mid P\in Y\}$, which is exactly the definition of $\mathscr{P}(Y)$.
\end{proof}

 We now introduce some notation that will be used in the following  Remark \ref{oss:flattop} and  Proposition \ref{chi}, where we will show that, given a ring $R$,   
$\xcal({R})$ is a topological retract of the spectral space $\scal(R(T))$.

 If $\mathscr{Q} \in \scal(R(T))$, then we set  $\Sigma_{\mathscr{Q} }:= R(T) \setminus \mathscr{Q} $,
 $ R(T)_{\mathscr{Q} } := \Sigma_{\mathscr{Q} }^{-1}R(T)$. 
 We denote by $g: R \rightarrow R(T)$ and $\lambda_1:  R(T) \rightarrow  R(T)_{\mathscr{Q} } $  the canonical flat homomorphisms and we set $ \lambda := \lambda_1 \circ g : R \rightarrow R(T)_{\mathscr{Q} }$. 
 
  \begin{oss} \label{oss:flattop} 
  In \cite{dofopa} the authors introduced and studied what they called the {\it  flat topology} on   $\spec( R)$, where $R$ is any ring, by taking as closed subspaces the subset $\rho^a(\spec(R'))$  for  $\rho: R \rightarrow R'$  varying among the flat ring homomorphisms. By \cite[Theorem 2.2]{dofopa} the flat topology on  $\spec( R)$ coincides with the inverse topology. 
  
  We are in condition to give an explicit description of the inverse-closed subspaces of $\spec({R})$.
  Let $Y \subseteq \spec({R})$, set as above $\boldsymbol{\ms Q}(Y) := \bigcup \{PR(T)\mid P\in Y\} \in \scal(R(T))$. 
  Then, it is straightforward to see  that $\boldsymbol{\ms Q}(Y)=\boldsymbol{\ms Q}(\lambda^a(\spec(R(T)_{\boldsymbol{\ms Q}(Y)})))$, 
  where $\lambda: R\rightarrow R(T)_{\boldsymbol{\ms Q}(Y)}$ is the canonical flat embedding. Thus, in view of Lemma \ref{gamma}(3) and of the fact that the image of $\lambda^a$ is closed in the inverse topology, being $\lambda$ flat, we have $\chiusinv(Y)=\lambda^a(\spec(R(T)_{\boldsymbol{\ms Q}(Y)}))$. In particular,    
 $ Y = \chiusinv(Y) $ if and only if   $Y = \lambda^a(\spec(R(T)_{\boldsymbol{\ms Q}(Y)}))$.
  \end{oss}
  
\begin{prop} \label{chi}
Let $R$ be   a ring. With the notation introduced above, the map

\begin{equation*}
\begin{aligned}
\boldsymbol{\chi}\colon \scal(R(T)) & \longrightarrow\xcal({R}) \\
\mathscr{Q} & \longmapsto \lambda^a(\spec(R(T)_{\mathscr{Q}}))
\end{aligned}
\end{equation*}
is continuous and surjective. 
Moreover, if $\boldsymbol{\nu}:\xcal({R}) \hookrightarrow \scal(R(T))  $ is
 the spectral embedding defined in Proposition \ref{nagata}(1), then $\boldsymbol{\chi}\circ\boldsymbol{\nu}$ is the identity on $\xcal({R})$.
\end{prop}
\begin{proof}
Note that $\boldsymbol{\chi}$ is well-defined by Remark \ref{oss:flattop}, since,   for any $\ms Q\in \scal(R(T))$, the canonical homomorphism $\lambda: R \rightarrow R(T)_{\mathscr{Q} }$ is flat.
 Let $\xcal(g^a) : \xcal(R(T)) \rightarrow \xcal(R)$ 
 be the spectral   map associated  to the canonical  flat ring embedding $g: R \rightarrow R(T)$ 
 and defined by  $\xcal(g^a)(Y):=   g^a(Y)^{\mbox{\tiny \texttt{gen}}}=g^a(Y) = \{g^{-1}(Q) \mid Q \in Y \}= \{Q \cap R \mid Q \in Y \}$, 
 for each $Y \in  \xcal(R(T))$ (see  \cite[Proposition 4.1]{FiFoSp-X(X)}  and \cite[Proposition 2.7]{dofopa}). 
 Then, the map  $\boldsymbol{\chi}$ coincides with the composition of the topological embedding 
$j: \scal(R(T))  \hookrightarrow \xcal(R(T)) $  (Proposition \ref{embedd-II}(1)) with $\xcal(g^a)$. 
In fact, 
\begin{equation*}
\begin{array}{rl}
 (\xcal(g^a) \circ j)(\mathscr{Q}) =&  \xcal(g^a)(\{Q \in  \spec(R(T)) \mid Q \subseteq \mathscr{Q}\}) \\
 =&  
\xcal(g^a)(\lambda^a(\spec(R(T))_{\mathscr{Q} })) =g^a(\lambda^a(\spec(R(T)_{\mathscr{Q} })))  \\
 =& \xcal(\lambda^a)(\spec(R(T)_{\mathscr{Q} })) = \lambda^a(\spec( R(T)_{\mathscr{Q} }) = \boldsymbol{\chi}( \mathscr{Q}))\,.
 \end{array}
\end{equation*}
Hence $\boldsymbol{\chi}$ is continuous as a composition of continuous maps  (Proposition \ref{embedd-II}(1) and \cite[Proposition 4.1]{FiFoSp-X(X)}).

Let now $Y\in\xcal({R})$.  Set,  as usual,  $\mathscr{Q}(Y) := \bigcup\{ PR(T) \mid P \in Y\}$. Then, a direct calculation shows that $(\boldsymbol{\chi}\circ\boldsymbol{\nu})(Y)$ is the  canonical image of $\spec(R(T)_{\mathscr{Q}(Y)})$  into  $\spec({R})$, which is is clearly equal to $Y$ (Remark \ref{oss:flattop}). Therefore $\boldsymbol{\chi}\circ\boldsymbol{\nu}$ is the identity. This  implies  that $\boldsymbol{\chi}$ is surjective.
\end{proof}

\begin{figure}
\begin{equation*}
\begin{tikzcd}[column sep=large]
\Spec(R)\arrow[hook]{r}{i}\arrow[bend left,hook]{rr}{\varphi} & \scal(R)\arrow[transform canvas={yshift=.5ex},hook]{r}{j_R} & \arrow[transform canvas={yshift=-0.5ex},two heads]{l}{\mathscr{P}_R}\xcal(R)\arrow[transform canvas={yshift=0.5ex},hook]{r}{\boldsymbol{\nu}} & \arrow[bend left,two heads]{ll}{\scal(g)}\scal(R(T))\arrow[transform canvas={yshift=-0.5ex},two heads]{l}{\boldsymbol{\chi}}\arrow[transform canvas={yshift=.5ex},hook]{r}{j_{R(T)}} & \arrow[transform canvas={yshift=-0.5ex},two heads]{l}{\mathscr{P}_{R(T)}}\arrow[bend right,two heads]{ll}{\xcal(g^a)}\xcal(R(T))
\end{tikzcd}
\end{equation*}
\caption{Maps between $\scal$- and $\xcal$-type spaces.}
\end{figure}

\begin{oss}\label{notchi}
 Given a ring $R$,  there is another possible    natural  way to define a   continuous  map $\scal(R(T))\longrightarrow\xcal(R)$. Indeed, define $\boldsymbol{\chi'}$ as the map
\begin{equation*}
\begin{aligned}
\boldsymbol{\chi'}\colon\scal(R(T)) & \longrightarrow \xcal(R)\\
\mathscr{Q} & \longmapsto \{P\in\Spec(R)\mid g(P)\subseteq\mathscr{Q}\}.
\end{aligned}
\end{equation*}
Clearly, $\boldsymbol{\chi}(\mathscr{Q}) \subseteq \boldsymbol{\chi'}(\mathscr{Q})$, for each $\mathscr{Q} \in \scal(R(T))$.  Moreover, a direct calculation shows that $\boldsymbol{\chi'}=j\circ\scal(g)$, so that $\boldsymbol{\chi'}$ is continuous.  Furthermore,  by Proposition \ref{nagata}(3), we have
\begin{equation*}
\boldsymbol{\chi'}\circ\boldsymbol{\nu}= j \circ\scal(g)\circ \boldsymbol{\nu}=j \circ\mathscr{P}.
\end{equation*}
Recall that $\boldsymbol{\chi}\circ\boldsymbol{\nu}$ is the identity on $\xcal({R})$  (Proposition \ref{chi}) and, in general, $\boldsymbol{\chi'}\circ\boldsymbol{\nu} \ (=j \circ\mathscr{P})$ is not  (Proposition \ref{homeo}(3)). We note that $\boldsymbol{\chi'}$, unlike $\boldsymbol{\chi}$, is not surjective:  for example, let $D$ be a 2-dimensional Noetherian local ring and let $\Spec^1(D)$ be the set of the height-1 primes   of $D$.
 Then,   $Z:=\Spec^1(D)\cup\{(0)\}$  is inverse-closed in $\Spec(D)$, and the maximal ideal $M$ of $D$ is contained in the union of the elements of   $Z$.
 Hence, $ MD(T)\subseteq\mathscr{Q}(Z)$, and thus $M\in\{P\in\Spec(D)\mid g(P)\subseteq\mathscr{Q}(Z)\}=\boldsymbol{\chi'}(\mathscr{Q}(Z))$. Therefore, $\boldsymbol{\chi'}(\mathscr{Q}(Z))=\Spec(D)$.
On the other hand, $M \notin\boldsymbol{\chi}(\mathscr{Q}(Z))$,  since  $Z=\boldsymbol{\chi}( \mathscr{Q}(Z))$, because $Z$ is inverse-closed (Remark \ref{oss:flattop}).
We easily conclude that  $Z$ is not in the range of $\boldsymbol{\chi'}$.
As a matter of fact, suppose there exists a semigroup prime $\ms Q^\star$ of $D(T)$ such that $Z=\boldsymbol{\chi'}(\ms Q^\star)=\{P\in \spec(D)\mid P\subseteq g^{-1}(\ms Q^\star) \}$. Thus, the union of all the prime ideals belonging to $Z$ is contained in $g^{-1}(\ms Q^\star)$ and, a fortiori, $M\subseteq g^{-1}(\ms Q^\star)$. It follows that $M\in  \boldsymbol{\chi'}(\ms Q^\star)=Z$, a contradiction.   
\end{oss}

\smallskip
\noindent  {\bf Acknowledgment.}  The authors thank the referee for his/her thorough report and highly appreciate the constructive comments and
suggestions, which contributed to improving the quality of the paper.






\end{document}